\documentclass[reqno,centertags,11pt]{amsart}

\usepackage{amssymb,amsmath,amsfonts,amssymb}%numbysec}
\usepackage{hyperref}
\usepackage{enumerate}

-\marginparwidth \oddsidemargin -9mm \evensidemargin \oddsidemargin
\usepackage{latexsym}
\advance\hoffset by 5mm

\def\@abssec#1{\vspace{.05in}\footnotesize \parindent .2in
{\bf #1. }\ignorespaces}

\newtheorem{theorem}{Theorem}[section]

\newtheorem{lemma}[theorem]{Lemma}
\newtheorem{proposition}[theorem]{Proposition}

\newtheorem{definition}[theorem]{Definition}

\DeclareMathOperator{\divg}{div}

\newcommand{\R}{\ensuremath{\mathbb{R}}}

\newcommand{\N}{\ensuremath{\mathbb{N}}}

\newcommand{\MM}{\ensuremath{\mathcal{M}}}

\newcommand{\Id}{\ensuremath{\mathrm{Id}}}

\newcommand{\set}[1]{\left\{#1\right\}}
\newcommand{\dd}{\mathrm{d}}

\allowdisplaybreaks \numberwithin{equation}{section}

\begin{document}

\title[Boussinesq system with measure forcing]{Boussinesq system with measure forcing}
\author{Piotr B. Mucha}
\address{Institute of Applied Mathematics and Mechanics, University of Warsaw, ul. Banacha 2 Warszawa, Poland}
\email{p.mucha@mimuw.edu.pl}
\author{Liutang Xue}
\address{School of Mathematical Sciences and LMSC of MOE, Beijing Normal University, Beijing 100875, P.R. China}
\email{xuelt@bnu.edu.cn}
\date{\today}
\maketitle

\begin{abstract}
We address a question concerning the issue of existence to a Boussinesq type system with a heat source. The problem is studied in the whole two dimensional
plane and the heat source is a measure transported by the flow. For arbitrary
initial data, we prove global in time existence of unique regular solutions.
Measure being a heat source limits regularity of constructing solutions and make us work in a non-standard framework of inhomogeneous Besov spaces of the
$L^\infty(0,T;B^s_{p,\infty})$-type. Application of the Lagrangian coordinates
yields uniqueness omitting difficulties with comparison of  measures.
\end{abstract}

\section{Introduction}

Heat conducting fluids are an important part of the fluid mechanics. In the highest generality they are complete from the viewpoint of the conservation of the total energy. Viscous fluids generate internal
friction and produce thermal effects, and vice versa variability of the temperature creates a motion of the fluid. In the general form we
distinguish the Navier-Stokes-Fourier model for the compressible flows:
\begin{equation}\label{CNS}
 \begin{cases}
  \partial_t \rho +\divg(\rho u)=0, \\
  \partial_t(\rho u) +\divg(\rho u \otimes u) -\divg S(\theta,\nabla u)+\nabla p(\rho, \theta)= \rho f,\\
  \partial_t\big(\rho s(\rho,\theta)\big) + \divg \big(\rho s(\rho,\theta)u\big)
  + \divg\big(\frac{q(\theta,\nabla \theta)}{\theta}\big) = \sigma.
 \end{cases}
\end{equation}
In short, $\rho$, $u$, $\theta$ are sought quantities: the density, velocity and the
temperature of the fluid. Functions $p(\cdot,\cdot)$ and $s(\cdot,\cdot)$ are the pressure and entropy.
The stress tensor $S$ is given in the Newtonian form and the energy  flux is given in the
Fourier form $q=-\kappa(\theta) \nabla \theta$ and the entropy production $\sigma=
\frac{1}{\theta}(S:\nabla u + \frac{\kappa(\theta) |\nabla \theta|^2}{\theta})$ (for more details see \cite{F}).

Nowadays mathematics is able to deliver just existence of weak solutions \cite{F,FN,FMNP} for the system (\ref{CNS}),
and regular solutions are just possible to get for small data \cite{MN,D,VZ}.
The system looks too much complex. It makes us look for a reduction of it. Taking a low Mach number limit
(see \cite{FN2,DM-mono}), we  obtain an incompressible limit which takes into account weak thermal effects, known as the Boussinesq approximation
\begin{equation}\label{BoussEq-1}
\begin{cases}
  \partial_t \theta + v\cdot\nabla \theta - \Delta \theta = f, \\
  \partial_t v + v\cdot\nabla v - \Delta v + \nabla p = \theta \,e_d, \\
  \mathrm{div}\, v=0,
\end{cases}
\end{equation}
where $d=2,3$, $v$ is the velocity vector field, $\theta$ is the temperature field, and $e_d$ is the last canonical vector of $\R^d$.
In the simplest explanation, the above system \eqref{BoussEq-1} is the incompressible Navier-Stokes equations coupled with the heat equation with a drift given by the velocity,
forcing for the momentum equation is defined by the change of temperature in the direction of the gravitational force (i.e. $e_d$-direction).
For the mathematical study of system \eqref{BoussEq-1} with $f=0$, one can see \cite{CanD,Guo,Temam,BS} for the global well-posedness results (with small data assumption in 3D case).

What is important to underline is the following fact, the system \underline{does not}
preserve the energy, in \cite{BS} the authors proved that for $f=0$ in system \eqref{BoussEq-1} the total energy $\|u(t)\|_{L^2}^2$ may grow in time. It makes our mathematical analysis more interesting. The dynamics is nontrivial for long time and most of norms of solutions are expected to growth in time.

Let us explain the goal of our paper. We want to consider a special case of system (\ref{BoussEq-1}) as the force is given by a heat source transported by the flow:
\begin{equation}\label{BoussEq}
\begin{cases}
  \partial_t \mu + v \cdot \nabla \mu =0, \\
  \partial_t \theta + v\cdot\nabla \theta - \Delta \theta = \mu, \\
  \partial_t v + v\cdot\nabla v - \Delta v + \nabla p = \theta \,e_d, \\
  \mathrm{div}\, v=0, \\
  (\mu,\theta,v)|_{t=0}(x)=(\mu_0,\theta_0,v_0)(x),
\end{cases}
\end{equation}
where the system is consider in the whole space $\R^d$.

In the most interesting case one can think that $\mu$ describes a combustion distributed by some measure like a linear combination of some Dirac atoms.
The physical explanation can be a modeling the movement of water after putting some chemical material like Sodium (Na) into a square pool fully containing water.

The main goal of the paper is to consider large solutions to construct
global in time solvability. Since the Millennium Problem concerning the
regularity of weak solutions to the three dimensional Navier-Stokes system is still open, we here concentrate our attention on the case of two spacial dimension.
The key point is to consider general data admitting initial heat production as a Radon measure and large initial data of velocity and temperature.
%The form of information about the force $\mu$ makes us work in a special Besov spaces of the type $L^\infty(0,T; B^s_{p,\infty}(\R^2))$.

%If $\mu_0\in \mathcal{M}(\T^d)$, we will get $\mu(t)\in \mathcal{M}\subset pm^0(\T^d)$, $\theta\in L^\infty(\R^+; pm^2(\T^d))$ and $v\in L^\infty(\R^+; pm^4(T^d))$.
%Roughly thinking, in the framework of Besov spaces, we expect that $\mu\in L^\infty_t \dot B^0_{1,\infty}$, and $\theta\in L^\infty_t \dot B^2_{1,\infty}$, and further $v\in L^\infty_t \dot B^4_{1,\infty}$.
%Since $\dot B^4_{1,\infty}(\T^3) \subseteq \dot B^1_{\infty,\infty}(\T^3)$ and $\dot B^4_{1,\infty}(\T^3) \nsubseteq \dot B^1_{\infty,\infty}(\T^3)$, it seems that in three dimension the velocity $v$ still will not be a Lipschitz flow, which means that the result may be lack of uniqueness.
%\vskip0.2cm
Our result is the global in time existence and uniqueness of regular solution for the system \eqref{BoussEq} at the 2D case.
\begin{theorem}\label{thm:2Dgwp}
    Let $\mu_0\in \mathcal{M}_+(\R^2)$ with $\mathrm{supp}\,\mu_0\subset B_{R_0}(0)$ for some $R_0>0$.
For each $\sigma\in ]0,2[$, let $\theta_0 \in L^1\cap B^{2-\sigma}_{\frac{4}{4-\sigma},\infty}(\R^2)$ with $\theta_0\geq 0$,
and $v_0\in H^1(\R^2)$ be a divergence-free vector field with initial vorticity $\omega_0=\partial_1 v_{2,0}-\partial_2 v_{1,0}\in B^{3-\sigma}_{\frac{4}{4-\sigma},\infty}(\R^2)$.
Let $T>0$ be any given. Then the system \eqref{BoussEq} admits a unique solution $(\mu,\theta,v)$ on $[0,T]$ such that
\begin{equation}\label{mu-concl}
  \mu\in L^\infty([0,T]; \mathcal{M}_+(\R^2)),\quad \textrm{with}\quad \mathrm{supp}\,\mu\subset B_{R_0+C_T},
\end{equation}
and
\begin{equation}\label{the-concl}
  \theta\in L^\infty([0,T]; L^1\cap B^{2-\sigma}_{\frac{4}{4-\sigma},\infty}(\R^2)),\quad\textrm{with}\quad \theta\geq 0 \;\;\mathrm{on}\;\;[0,T]\times\R^2,
\end{equation}
and
\begin{equation}\label{v-concl}
  v\in L^\infty([0,T]; H^1(\R^2))\cap L^2([0,T]; H^2(\R^2))\cap L^\infty([0,T]; W^{1,\infty}(\R^2)),
\end{equation}
where $C_T>0$ is a constant depending on $T$ and the norms of initial data.
\end{theorem}

The above statement requires some explanation. In general the global wellposedness of system (\ref{BoussEq}) with smooth forcing $\mu$ in the two spacial dimension case for large data is clear. Thanks to the famous result of
Ladyzhenskaya \cite{Lad} (and in the language of Besov space \cite{DM-ADV}), we are able to obtain the regular solutions to the Navier-Stokes equations. The system (\ref{BoussEq}) from the
regularity viewpoint is a relatively weak perturbation and basic energy norms grant us the standard existence at  the level of Galerkin's method,
but only for a suitable approximation related with smoothed out initial data -- see details in the subsection \ref{subsec:exe}. However this approach
works just for smooth forcing $\mu$, and for the original initial data we are required to proceed in a non-standard way. To avoid technical problem with definition of measures at infinity we assume that the initial heat source is
compactly supported in space.
\vskip0.1cm

Our result has three interesting ingredients:

\smallskip

$\ast$ \ The first one, the heat source is just a measure which is not vanishing in time, and the only information one can get here is the $L^\infty$
in time and measure in space (see Proposition \ref{prop:ape1}). It requires quite high regularity of the velocity, indeed the Lipschitz continuity, to guarantee the existence and uniqueness. On the other hand we can not expect too much regular solutions since they are generated by a measure forcing. Nevertheless, our solutions are regular and $(\ref{BoussEq})_1$ solves $\mu$ in terms of characteristics, which are well-defined.

\smallskip

$\ast$ \ The second one, it is an application of non-standard Besov spaces $L^\infty(0,T; B^s_{p,\infty}(\R^2))$ of first time to consider the measure force.
The basic restriction here is that the measure $\mu$ will belong to $L^\infty_T(B^{-\sigma}_{\frac{4}{4-\sigma},\infty})$, which is slightly larger than the expected one.
Such a framework fits perfectly to the regularity properties of the right-hand side of equation $(\ref{BoussEq})_2$.

\smallskip

$\ast$ \ The last one is the limited regularity of solution. In the construction of the \textit{a priori} estimates, it appears that the force $\mu$ given as a measure does not allow to use just standard bounds by
the energy norms. In showing the crucial $L^2$-estimate of velocity $v$ and vorticity $\omega$, we have to control $\|\theta(t)\|_{L^2}$ which is not so direct due to the effect of the measure force; we have a natural uniform $L^1$-bound $\|\theta\|_{L^\infty_t(L^1)}$,
and for the higher regularity we use Lemma \ref{lem:heatEq} to derive an estimate of quantity $\|\theta\|_{L^\infty_t (B^{2-\sigma}_{\frac{4}{4-\sigma},\infty})}$ in terms of $\|(v,\omega)\|_{L^\infty_t (L^2)}$, and then the needing estimate $\|\theta(t)\|_{L^2}$ is bounded from interpolation of these two quantities. Somehow we consider here a limit case and the final estimate of $\|(v,\omega)\|_{L^\infty_t (L^2)}$ (see \eqref{eq:vEs-key0} below) is obtained by an application of a new logarithmic interpolation inequality for the Besov spaces (see Lemma \ref{lem:intp}).
\vskip0.1cm

Besides, we also point out that since the measure force $\mu$ is determined via a transport equation, it seems not sufficient to show the uniqueness in the framework of Eulerian coordinates and we have to adapt the Lagrangian coordinates
(see \cite{DanM12,DanM13} and references therein for this novel method used in the density-dependent incompressible Navier-Stokes equations).
The limited regularity of velocity and temperature field also makes much difficulty in considering the difference system \eqref{L-BEq-del} by using the standard $L^2$-energy estimates,
but instead we work on a non-standard setting, that is, we consider the $\dot H^{-1}$-estimate of $\delta \bar{\theta}$ and $\dot H^1$-estimate of $\delta \bar{v}$ simultaneously,
and by a series of energy type estimates in the Lagrangian coordinates we manage to show the uniqueness.
\vskip0.1cm

The outline of this paper is as follows. We present preliminary results including some auxiliary lemmas in subsection \ref{sec:prelim}. We give the detailed proof of Theorem \ref{thm:2Dgwp} in the whole section \ref{sec:thm-2dgwp}:
we firstly show the key \textit{a priori} estimates of solution $(\mu,\theta,v)$ in subsection \ref{subsec:apE}, then we sketch the proof of existence in subsection
\ref{subsec:exe}, and finally we prove the uniqueness by using Lagrangian coordinates in subsection \ref{subsec:uniq}. In the last appendix section we show the proof of Lemma \ref{lem:prodEs}. %used in deriving \textit{a priori} estimates. %concerning the product estimates in Besov spaces.

\section{Preliminaries}\label{sec:prelim}

In this section, some notations are listed, and we compile basic results related to measure and Lagrangian coordinates, and also show some auxiliary lemmas used in the paper.

The following notations are used throughout this paper.

\noindent $\diamond$ $C$ stands for a constant which may be different from line to line, and $C(\lambda_1,\cdots,\lambda_n)$ denotes a constant $C$ depending on the coefficients $\lambda_1,\cdots,\lambda_n$.
%$X\lesssim Y$ means that there is a harmless constant $C$ such that $X\leq C Y$, and $X\approx Y$ means that $X\lesssim Y$ and $Y\lesssim X$ simultaneously.

\noindent $\diamond$ The notation $\mathcal{D}(\R^d )$ or $\mathcal{D}(\R^d\times[0,T])$ denotes the space of $C^\infty$-smooth functions with compact support on $\R^d$ or $\R^d\times [0,T] $, respectively.
$\mathcal{D}'(\R^d\times[0,T])$ is the space of distribution on $\R^d\times [0,T] $ which is the dual space of $\mathcal{D}(\R^d\times[0,T])$.

\noindent $\diamond$ The notation $\mathcal{S}(\mathbb{R}^d )$ is the Schwartz class of rapidly decreasing $C^\infty$-smooth functions, and $\mathcal{S}'(\mathbb{R}^d )$ is the space of tempered distributions which is the dual space of $\mathcal{S}(\R^d)$.

\noindent $\diamond$ For $m\in\N$, $r\in [1,+\infty]$, $s\in \R$, we denote by $W^{m,r}(\R^d )$ ($\dot W^{m,r}(\R^d)$) and $H^s(\R^d )$ ($\dot H^s(\R^d )$)
the usual $L^r$-based and $L^2$-based inhomogeneous (homogenous) Sobolev spaces.

\noindent $\diamond$ For Banach space $X=X(\R^d)$ and $\rho\in [1,\infty]$, the notation $L^\rho(0,T; X)$ denotes the usual space-time space $L^\rho([0,T]; X)$, which is also abbreviated as $L^\rho_T(X)$.

\noindent $\diamond$ We use $B_r(x_0):=\{x\in \mathbb{R}^d: |x-x_0|< r\}$ to denote the open ball of $\mathbb{R}^d$.

%\noindent $\diamond$ We use $\mathcal{F}(f)$ (or $\widehat{f}$) and $\mathcal{F}^{-1}(f)$ to denote the Fourier transform and the inverse Fourier transform of a function $f$, that is, $\mathcal{F}(f)(\zeta)=\int_{\mathbb{R} }e^{i x\cdot \zeta} f(x)\mathrm{d} x$ and $\mathcal{F}^{-1}(g)(x) = \frac{1}{2\pi}\int_{\R} e^{i x \cdot\zeta} g(\zeta)\dd \zeta $.

\subsection{Results related to measure}

We denote $\mathcal{M}=\mathcal{M}(\R^d)$ as the space of finite Radon measures defined on $\R^d$ with \emph{total variation} topology,
i.e., for any $\mu$ Radon measure, define
\begin{equation*}
  \|\mu\|_{\mathcal{M}(\R^d)}=|\mu|(\R^d):=\sup\left\{\left|\int_{\R^d} f \dd \mu \right|: \|f\|_{L^\infty}\leq 1, f\in C_0(\R^d) \right\}.
\end{equation*}
As a consequence of Riesz representation theorem, $\mathcal{M}(\R^d)$ is the dual space of $C_0(\R^d)$.
For $\{\mu_n\}_{n\in\N}\subset \MM(\R^d)$ and $\mu\in \MM(\R^d)$, we say $\mu_n\rightarrow \mu$ weakly$*$ if
\begin{equation*}
  \lim_{n\rightarrow \infty} \int_{\R^d} f \dd \mu_n = \int_{\R^d} f \dd \mu, \quad \textrm{for every  }f\in C_0(\R^d).
\end{equation*}

We also denote $\mathcal{M}=\mathcal{M}(\R^d,d)$ as the space of finite Radon measures on $\R^d$ equipped with \emph{bounded Lipschitz distance} topology,
i.e., for any Radon measures $\mu$ and $\nu$, define
\begin{equation}\label{d-topology}
  d(\mu,\nu):= \sup\left\{\left|\int_{\R^d} f \dd \mu - \int_{\R^d} f \dd \nu \right|: \|f\|_{L^\infty}\leq 1\; \textrm{and}\;
  Lip(f):= \sup_{x\neq y\in\R^d} \frac{|f(x)-f(y)|}{|x-y|} \leq 1\right\} .
\end{equation}

We denote $\mathcal{M}_+= \mathcal{M}_+(\R^d)$ the set of nonnegative finite Radon measures on $\R^d$, i.e.,
\begin{equation*}
  \mathcal{M}_+ := \left\{\mu\in \mathcal{M}(\R^d): \mu\geq 0\right\},
\end{equation*}
both with the strong total variation and  weak $d(\cdot,\cdot)$ topologies.

\begin{definition}
  We say that a sequence $\{\mu_n\}\subset\mathcal{M}(\R^d)$ is \emph{tight} if for any $\epsilon>0$,
there exists a compact set $K_\epsilon \subset\subset\R^d$ so that
\begin{equation*}
  \sup_{n\in \N} |\mu_n|(\R^d\setminus K_\epsilon) <\epsilon,
\end{equation*}
where $|\mu_n|$ is the total variation measure of $\mu_n$.
\end{definition}

\begin{proposition}[cf. Theorem 2.7 of \cite{GLM}]
  Let $\{\mu_n\}_{n\in\N}$ be a tight sequence in $\MM(\R^d)$ and let $\mu\in \MM(\R^d)$.
Then as $n\rightarrow \infty$, $\mu_n\rightarrow \mu$ weakly$*$ if and only if $d(\mu_n,\mu)\rightarrow 0$ and $\sup_{n\in\N} |\mu_n|(\R^d)<\infty$.
\end{proposition}

The space $(\mathcal{M}_+,d)$ is a complete metric space.

\begin{proposition}[cf. Corollary 21 of \cite{MP18}]\label{pro:mu-conv}
  Let $\{\mu_n\}_{n\in\N}$ be a sequence bounded in $\mathcal{M}_+(\R^d)$ with support contained in a given ball.
Then there exists a $(\mathcal{M}_+,d)$-convergent subsequence $\{\mu_{n_k}\}$.
\end{proposition}

\subsection{Auxiliary lemmas}

%{\bf Piotr: here we need to put the definition of Besov spaces, inhomogeneous one}.

Before presenting some auxiliary lemmas used in this paper, we recall the definitions of nonhomogeneous Besov spaces and their space-time counterparts.
One can choose two nonnegative radial functions $\chi, \varphi\in \mathcal{D}(\mathbb{R}^d)$ be
supported respectively in the ball $\{\xi\in \mathbb{R}^d:|\xi|\leq \frac{4}{3} \}$ and the annulus $\{\xi\in
\mathbb{R}^d: \frac{3}{4}\leq |\xi|\leq  \frac{8}{3} \}$ such that (e.g. see \cite{Danc})
\begin{equation*}
  \chi(\xi)+\sum_{j\in \mathbb{N}}\varphi(2^{-j}\xi)=1, \quad \forall \xi\in \mathbb{R}^d.
\end{equation*}
For every $ f\in S'(\R^d)$, we define the non-homogeneous Littlewood-Paley operators as follows
\begin{equation}\label{LPop}
  \Delta_{-1}f:=\chi(D)f; \quad \, \quad\Delta_{j}f:=\varphi(2^{-j}D)f,\;\;\;S_j f:=\sum_{-1\leq k\leq j-1} \Delta_{k}f,\;\;\;\forall j\in \mathbb{N}.
\end{equation}
Now for $s\in \mathbb{R}, (p,r)\in[1,+\infty]^2$, the inhomogeneous Besov space $B_{p,r}^s=B^s_{p,r}(\R^d)$ is defined as
\begin{equation*}
  B^s_{p,r}:=\Big\{f\in\mathcal{S}'(\mathbb{R}^d);\|f\|_{B^s_{p,r}}:=\|\{2^{js}\|\Delta
  _j f\|_{L^p}\}_{j\geq -1}\|_{\ell^r }<\infty  \Big\}.
\end{equation*}
The classical space-time Besov space $L^{\rho}([0,T],B^{s}_{p,r})$, abbreviated by
$L^{\rho}_{T}(B^{s}_{p,r})$, is the set of tempered distribution $f$
such that
\begin{equation*}
  \|f\|_{L^{\rho}_{T}(B^{s}_{p,r})}:=\big\|\|\{2^{js}\|\Delta_{j}f\|_{L^{p}}\}_{j\geq -1}\|_{\ell^{r}}\big\|_{L^{\rho}([0,T])}<\infty.
\end{equation*}
Another one is the Chemin-Lerner's mixed space-time Besov space $\widetilde{L}^{\rho}([0,T],B^{s}_{p,r})$, abbreviated by $\widetilde{L}^{\rho}_T (B^s_{p,r})$, which is the set of tempered distribution $f$ satisfying
\begin{equation*}
  \|f\|_{\widetilde L^\rho_T (B^s_{p,r})}:= \|\{2^{qs}\|\Delta_q f\|_{L^\rho_T (L^p)}  \}_{q\geq -1}\|_{\ell^r}<\infty.
\end{equation*}

Then we have the following regularity estimates of the heat equation in the framework of Besov spaces (see \cite[Theorem 2.2.5]{Danc}).
\begin{lemma}\label{lem:heatEq}
  Let $s\in \R$ and $1\leq \rho,\rho_1,r\leq \infty$. Let $T>0$, $u_0\in B^s_{p,r}(\R^d)$, and $f\in \widetilde{L}^\rho_T(B^{s-2+2/\rho}_{p,r})$.
Then the following nonhomogeneous heat equation
\begin{equation*}
  \partial_t u -\Delta u = f,\qquad u|_{t=0}(x)=u_0(x),\quad x\in \R^d,
\end{equation*}
has a unique solution $u$ in $\widetilde{L}^\rho_T(B^{s+2/\rho}_{p,r})\cap \widetilde{L}^\infty_T(B^s_{p,r})$ and there exists a constant $C=C(d)$ such that for all $\rho_1\in [\rho,\infty]$,
\begin{equation*}
  \|u\|_{\widetilde{L}^{\rho_1}_T(B^{s+ \frac{2}{\rho_1}}_{p,r}(\R^d))} \leq C \left((1+T^{\frac{1}{\rho_1}})\|u_0\|_{B^s_{p,r}(\R^d)} + (1+ T^{1+ \frac{1}{\rho_1}-\frac{1}{\rho}})\|f\|_{\widetilde{L}^\rho_T(B^{s-2+ \frac{2}{\rho}}_{p,r}(\R^d))} \right).
\end{equation*}
In particular, for $\rho=\rho_1=r=\infty$, we have
\begin{equation}\label{eq:heat-ape}
  \|u\|_{L^\infty_T(B^s_{p,\infty}(\R^d))} \leq C \left(\|u_0\|_{B^s_{p,\infty}(\R^d)} + (1+ T) \|f\|_{L^\infty_T(B^{s-2}_{p,\infty}(\R^d))} \right).
\end{equation}
\end{lemma}

In obtaining the \textit{a priori} estimates of the main theorem, we use the following product estimates in Besov spaces (whose proof is put to the appendix section).
\begin{lemma}\label{lem:prodEs}
Let $v:\R^2\rightarrow\R^2$ be a divergence-free vector field and $\theta:\R^2\rightarrow \R$ be a scalar function.
\begin{enumerate}[(1)]
\item
Let $s\in ]0,1[$, $ p\in [1,\infty]$. Then there exists a positive constant $C=C(s)$ such that
\begin{equation}\label{eq:prodEs}
  \|v\cdot\nabla \theta\|_{B^{-s}_{p,\infty}(\R^2)} \leq C \left(\|v\|_{L^2(\R^2)} + \|\nabla v\|_{L^2(\R^2)}\right) \Big(\sup_{k\geq -1} 2^{k(1-s)} \sqrt{k+2}\|\Delta_k\theta\|_{L^p(\R^2)}\Big).
\end{equation}
\item
Let $s\in ]0,1[$, $p\in [1,2]$.  Then there exists an absolute positive constant $C_0$ such that
\begin{equation}\label{eq:prodEs2}
\begin{split}
  \|v\cdot\nabla \theta\|_{B^s_{p,\infty}(\R^2)} & \leq C_0 \left(\|v\|_{L^{2p}(\R^2)} \|\theta\|_{B^{1+s}_{2p,\infty}(\R^2)} + \|v\|_{B^s_{2p,\infty}(\R^2)} \|\nabla\theta\|_{L^{2p}(\R^2)} \right) \\
  & \leq C \|v\|_{B^s_{2p,\infty}(\R^2)} \|\theta\|_{B^{1+s}_{2p,\infty}(\R^2)}.
\end{split}
\end{equation}
\item
Let $s\in ]0,\infty[$, $p\in [1,\infty]$. Then there exists a positive constant $C=C(s)$ such that
\begin{equation}\label{eq:prodEs3}
  \|v\cdot\nabla v\|_{B^s_{p,\infty}(\R^2)} \leq C\big( \|\nabla v\|_{B^s_{p,\infty}(\R^2)} \|v\|_{L^\infty(\R^2)} + \|v\|_{B^s_{p,\infty}(\R^2)} \|\nabla v\|_{L^\infty(\R^2)}\big).
\end{equation}
\end{enumerate}
\end{lemma}

We also have the following interpolation inequality dealing with the term appearing in the right-hand side of \eqref{eq:prodEs}.
%{\bf Piotr: plz check the statement i did small correction $d \sim 2$ }
\begin{lemma}\label{lem:intp}
  Let $s\in ]0,1[$, $p\in [1,\infty]$, and $\theta: \R^d\rightarrow \R$ be a scalar function. Then there is a positive constant $C=C(s,p,d)$ such that
\begin{equation}\label{eq:intp}
  \sup_{k\geq -1} 2^{k(1-s)} \sqrt{k+2}\|\Delta_k\theta\|_{L^p(\R^d)} \leq C \|\theta\|_{L^1(\R^d)}^{\frac{1}{2+d-s-d/p}} \|\theta\|_{B^{2-s}_{p,\infty}(\R^d)}^{\frac{1+d-s-d/p}{2+d-s-2/p}}
  \sqrt{\log\Big(e + \frac{\|\theta\|_{B^{2-s}_{p,\infty}}}{\|\theta\|_{L^1}}\Big)} + C\|\theta\|_{L^1}.
\end{equation}
\end{lemma}

\begin{proof}[Proof of Lemma \ref{lem:intp}]
  Let $N\in \N\cap [2,\infty[$ be an integer chosen later, then by using Bernsteins's inequality we have
\begin{align*}
  \sup_{k\geq -1} 2^{k(1-s)} & \sqrt{k+2}\|\Delta_k\theta\|_{L^p(\R^d)} \leq \sup_{-1\leq k\leq N} 2^{k(1-s)} \sqrt{k+2}\|\Delta_k\theta\|_{L^p} + \sup_{k\geq N} 2^{k(1-s)} \sqrt{k+2}\|\Delta_k\theta\|_{L^p} \\
  & \leq C_0 \sup_{-1\leq k\leq N} 2^{k(1-s)} \sqrt{2+k} 2^{k(d-\frac{d}{p})}\|\Delta_k\theta\|_{L^1} + C_0\sup_{k\geq N} \sqrt{2 +k} 2^{-k} 2^{k(2-s)}\|\Delta_k\theta\|_{L^p} \\
  & \leq C_0 2^{N(1+d-s-\frac{d}{p})} \sqrt{N} \|\theta\|_{L^1(\R^d)} + C_0 2^{-N} \sqrt{N} \|\theta\|_{B^{2-s}_{p,\infty}(\R^d)}.
\end{align*}
Now we define the constant $N$ as
\begin{equation}\label{eq:N}
  N:=
  \begin{cases}
    2,\quad & \textrm{if   }\quad \|\theta\|_{B^{2-s}_{p,\infty}(\R^d)} \leq 2\|\theta\|_{L^1(\R^d)}, \\
    \Big[\frac{1}{2+d-s -d/p}\log \Big(\frac{\|\theta\|_{B^{2-s}_{p,\infty}(\R^d)}}{\|\theta\|_{L^1(\R^d)}}\Big)\Big]+1,
    \quad & \textrm{if   }\quad \|\theta\|_{B^{2-s}_{p,\infty}(\R^d)} \geq 2\|\theta\|_{L^1(\R^d)},
  \end{cases}
\end{equation}
where notation $[a]$ means the integer part of $a\in\R$, then it is clear that the desired inequality \eqref{eq:intp} is followed by a direct computation.
\end{proof}

The following $L^2$-based estimate on the linear Stokes system is useful in the uniqueness proof.
\begin{lemma}[cf. Lemma 3 of \cite{DanM13}]\label{lem:Stokes}
  Let $R$ be a vector field satisfying $R_t\in L^2(\R^d\times ]0,T])$ and $\nabla \divg R\in L^2(\R^d\times ]0,T])$. Then the following system
\begin{equation}
\begin{cases}
  \partial_t u - \Delta u + \nabla P = f,\quad & \mathrm{in}\;\; \R^d\times ]0,T], \\
  \divg u =\divg R, \quad & \mathrm{in}\;\; \R^d\times ]0,T], \\
  u|_{t=0}= u_0,\quad & \mathrm{on}\;\; \R^d,
\end{cases}
\end{equation}
admits a unique solution $(u,\nabla P)$ which satisfies that
\begin{equation}
  \|\nabla u\|_{L^\infty_T (L^2)} + \|(u_t, \nabla^2 u,\nabla P)\|_{L^2_T (L^2)} \leq C \Big( \|\nabla u_0\|_{L^2} + \|(f,R_t)\|_{L^2_T(L^2)} + \|\nabla\divg R\|_{L^2_T (L^2)} \Big),
\end{equation}
where $C$ is a positive constant independent of $T$.
\end{lemma}

\subsection{The Lagrangian coordinates}\label{subsec:lag}

The use of Lagrange coordinates plays a fundamental role in the proof of the uniqueness part.
In this subsection, we introduce some notations and basic results related to the Lagrangian coordinates.

Let $X_v(t,y)$ solve the following ordinary differential equation (treating $y$ as a parameter)
\begin{equation}\label{flow}
  \frac{\dd X_v(t,y)}{\dd t} = v (t, X_v(t,y)),\quad X_v(t,y)|_{t=0}=y,
\end{equation}
which directly leads to
\begin{equation}\label{flow2}
  X_v(t,y) = y + \int_0^t v(\tau, X_v(\tau,y))\dd \tau.
\end{equation}

We list some basic properties for the Lagrangian change of variables.
\begin{lemma}\label{lem:Lag}
  Assume that $v\in L^1(0,T; \dot  W^{1,\infty}(\R^d))$. Then the system \eqref{flow} has a unique solution $X_v(t,y)$ on the time interval $[0,T]$
satisfying $\nabla_y X_v\in L^\infty(0,T; L^\infty)$ with
\begin{equation}\label{DXvest}
  \|\nabla_y X_v(t)\|_{L^\infty(\R^d)} \leq \exp \set{\int_0^t \|\nabla_x v(\tau)\|_{L^\infty(\R^d)}\dd \tau}.
\end{equation}
Furthermore, denoting by $\bar v(t,y):= v(t,X_v(t,y))$, we have
\begin{equation}\label{Xv}
  X_v(t,y)=y+ \int_0^t\bar v(\tau,y)\dd\tau,
\end{equation}
so that
\begin{equation}\label{DXveq}
  \nabla_y X_v(t,y)= \mathrm{Id} + \int_0^t \nabla_y \bar v(\tau,y)\dd \tau.
\end{equation}
Let $Y(t,\cdot)$ be the inverse diffeomorphism of $X(t,\cdot)$, then
$\nabla_x Y_v(t,x)= \left( \nabla_y X_v(t,y)\right)^{-1}$ with $x=X_v(t,y)$,
and if
\begin{equation}\label{vbar-cd}
  \int_0^t\|\nabla_y \bar v(\tau)\|_{L^\infty(\R^d)}\dd \tau\leq \frac{1}{2},
\end{equation}
we have
\begin{equation}\label{Avbd1}
  |\nabla_x Y_v(t,x) -\Id|\leq 2 \int_0^t |\nabla_y \bar v(\tau, y)|\dd \tau.
\end{equation}
%Finally, if $v\in L_1(0,T;  W^s_p(\R^d))$ with $s>\frac{n}{p}+1$, then $\nabla_y X_v-\Id\in L^\infty(0,T; W^{s-1}_p(\R^d))$.
\end{lemma}

\begin{proof}[Proof of Lemma \ref{lem:Lag}]
The proof is standard, and one can refer to \cite[Proposition 1]{DanM13} for details. We only note that as long as $\nabla_y X_v -\Id=\int_0^t \nabla_y \bar v(\tau,y)\dd \tau$
is sufficiently small so that \eqref{vbar-cd} holds, we have
\begin{equation}\label{DxY}
  \nabla_x Y_v = \left( \Id + (\nabla_y X_v -\Id)\right)^{-1} = \sum_{k=0}^\infty (-1)^k \left( \int_0^t \nabla_y \bar v(\tau,y)\dd \tau\right)^k,
\end{equation}
which immediately leads to \eqref{Avbd1}.
\end{proof}

Under the assumption $v\in L^1(0,T; \dot  W^{1,\infty}(\R^d))$, and using the Lagrangian coordinates introduced as above, we set
\begin{equation}
\begin{split}
  \bar{\mu}(t,y):=\mu(t, X_v(t,y)),\quad \bar{\theta}(t,y):= \theta(t, X_v(t,y)),\quad \overline{p}(t,y):= p(t,X_v(t,y)),
\end{split}
\end{equation}
then according to the deduction as in \cite{DanM12} or \cite{DanM13}, the Boussinesq type system \eqref{BoussEq} recasts in
\begin{equation}\label{L-BoussEq}
\begin{cases}
  \partial_t \bar{\mu} =0, \\
  \partial_t \bar{\theta} -\divg\left( A_v A_v^{\textrm{T}} \nabla_y \bar{\theta}\right)  = \bar \mu, \\
  \partial_t \bar{v} -\divg\left( A_v A_v^{\textrm{T}} \nabla_y \bar{v}\right)  + A_v^{\textrm{T}}\nabla_y \overline{p} = \bar{\theta} \,e_d, \\
  \divg_y\left( A_v \bar{v} \right) =0, \\
  \bar{\mu}|_{t=0}= \mu_0,\quad \bar{\theta}|_{t=0}= \theta_0,\quad \bar{v}|_{t=0}=v_0,
\end{cases}
\end{equation}
where we have adopted the notation
\begin{equation}\label{Av}
  A_v(t,y):=(\nabla_y X_v(t,y))^{-1}.
\end{equation}

As pointed out by \cite{DanM12,DanM13}, under the condition \eqref{vbar-cd}, the system \eqref{L-BoussEq} in the Lagrangian coordinates is equivalent to the system \eqref{BoussEq} in the Eulerian coordinates.
\vskip0.2cm

The first equation of \eqref{L-BoussEq} guarantees
\begin{equation}\label{bar-mu}
  \bar{\mu}(t,y)\equiv \mu_0(y),\quad \forall t\in [0,T],
\end{equation}
thus the system \eqref{L-BoussEq} reduces to
\begin{equation}\label{L-BoussEq2}
\begin{cases}
  \partial_t \bar{\theta} -\divg\left( A_v A_v^{\textrm{T}} \nabla_y \bar{\theta}\right)  = \mu_0, \\
  \partial_t \bar{v} -\divg\left( A_v A_v^{\textrm{T}} \nabla_y \bar{v}\right)  + A_v^{\textrm{T}}\nabla_y \overline{p} = \bar{\theta} \,e_d, \\
  \divg_y\left( A_v \bar{v} \right) =0, \\
  \bar{\theta}|_{t=0}= \theta_0,\quad \bar{v}|_{t=0}=v_0.
\end{cases}
\end{equation}

\section{Proof of Theorem \ref{thm:2Dgwp}}\label{sec:thm-2dgwp}

%The outline of the whole proof is as follows: in subsection \ref{subsec:apE} we show the key \textit{a priori} estimates, then we sketch the proof of existence in subsection
%\ref{subsec:exe}, and finally we prove the uniqueness result in subsection \ref{subsec:lag}.

\subsection{\textit{A priori} estimates}\label{subsec:apE}

\begin{proposition}\label{prop:ape1}
  Let $\mu_0\in \mathcal{M}_+(\R^2)$ be satisfying that $\mathrm{supp}\,\mu_0\subset B_{R_0}(0)$ for some $R_0>0$.
Let $T>0$ be any given, and $(\mu, \theta,v)$ be smooth functions on $\R^2\times [0,T]$ solving the system \eqref{BoussEq}.
Then for every $t\in[0,T]$, we have $\mu(t,x)=\mu_t(x)\in \mathcal{M}_+(\R^2)$ with
\begin{equation}\label{mu-es}
  \|\mu_t\|_{\mathcal{M}(\R^2)} \leq \|\mu_0\|_{\mathcal{M}(\R^2)},\quad \forall t\in[0,T],
\end{equation}
and also $\mathrm{supp}\,\mu_t \subset B_{R_0+C}(0)$ with $C=\|v\|_{L^1_T (L^\infty)}$.
\end{proposition}

\begin{proof}[Proof of Proposition \ref{prop:ape1}]
  Let $X_t(y)=X(t,y)$ be the flow function generated by the velocity $v$, which solves equation \eqref{flow} or \eqref{flow2}.
%\begin{equation}\label{eq:flowEq}
%  \frac{\dd X_t(x)}{\dd t} = v(t,X_t(x)),\quad X_0(x)=x,
%\end{equation}
%that is,
%\begin{equation}\label{eq:flowEq2}
%  X_t(x) = x + \int_0^t v(\tau, X_\tau(x))\dd \tau.
%\end{equation}
Since we assume that $v\in L^1([0,T]; W^{1,\infty}(\R^2))$, from Lemma \ref{lem:Lag}, it admits a unique vector field $X_t:\R^2\rightarrow \R^2$, $t\in [0,T]$ which is a diffeomorphism.

Let $Y_t=Y(t,\cdot)$ be the inverse diffeomorphism of $X_t$, then we see that
\begin{equation}\label{muExp}
  \mu(t,x)= \mu_t(x) = \mu_0(Y_t(x)).
\end{equation}
Clearly, $\mu_t\geq 0$, and since $X_t$ is volume-preserving (from the divergence-free property of $v$), we have
\begin{equation*}
\begin{split}
  \|\mu_t\|_{\mathcal{M}(\R^2)}= \sup_{\|g\|_{L^\infty}\leq 1} \left| \int_{\R^2} g(x) \dd \mu_t(x)\right|
  & =\sup_{\|g\|_{L^\infty}\leq 1} \left| \int_{\R^2} g(x) \dd \mu_0(Y_t(x))\right| \\
  & =\sup_{\|g\|_{L^\infty}\leq 1} \left| \int_{\R^2} g(X_t(y)) \dd \mu_0(y)\right| \\
  & \leq \sup_{\|\tilde{g}\|_{L^\infty}\leq 1} \left| \int_{\R^2} \tilde{g}(y) \dd \mu_0(y)\right|= \|\mu_0\|_{\mathcal{M}(\R^2)},
\end{split}
\end{equation*}
where the supremum is taken over all $C_0(\R^2)$ fuctions.

From \eqref{muExp} and $\mathrm{supp}\,\mu_0\subset B_{R_0}(0)$, we get $\mathrm{supp}\,\mu_t\subset X_t(B_{R_0}(0))$, and thus formula \eqref{flow2} implies that $\mathrm{supp}\,\mu_t \subset B_{R_0 + \|v\|_{L^1_T (L^\infty)}}(0)$.
%We then conclude Proposition \ref{prop:ape1}.
\end{proof}

\begin{proposition}\label{prop:ape2}
  Let $\mu_0\in \mathcal{M}_+(\R^2)$ with $\mathrm{supp}\,\mu_0\subset B_{R_0}(0)$ for some $R_0>0$, and $\theta_0 \in L^1(\R^2)$ with $\theta_0\geq 0$.
For $T>0$ any given, let $(\mu, \theta,v)$ be smooth functions on $\R^2\times [0,T]$ solving system \eqref{BoussEq}.
Then we have that $\theta(t)\geq 0$ for every $t\in[0,T]$ and also
\begin{equation}\label{theL1es}
  \sup_{t\in[0,T]}\|\theta(t)\|_{L^1(\R^2)}\leq \|\theta_0\|_{L^1(\R^2)} + T \|\mu_0\|_{\mathcal{M}(\R^2)}.
\end{equation}
\end{proposition}

\begin{proof}[Proof of Proposition \ref{prop:ape2}]
  We first prove the nonnegativity property of $\theta(t)$. The proof is standard (e.g. see \cite{LR10}) and it uses a contradiction argument.
Denote by $\Omega_T:= ]0,T]\times \R^2$. We define $\tilde{\theta}(t,x)= \theta(t,x) e^{-t}$ and assume that there is a constant $\lambda>0$ so that
\begin{equation*}
  \inf_{(t,x)\in \Omega_T} \tilde{\theta}(t,x)= -\lambda.
\end{equation*}
Such a constant $\lambda$ exists since we assume $\tilde\theta$ is a bounded smooth function. We also infer that there exists some point $(t_*,x_*)\in \Omega_T$ attaining this infimum.
Indeed, if not, there exists a sequence of points $(t_n,x_n)_{n\in\N}$ becoming unbounded such that $\tilde\theta(t_n,x_n)\rightarrow -\lambda$
as $n\rightarrow \infty$, which is a contradiction with the assumption that $\tilde\theta$ is a smooth function with suitable spatial decay.
%Indeed, since $\tilde\theta(t,x)$ is a smooth function with spatial decay, we have that for every $t\in]0,T]$, there exists some.

From the equation of $\tilde\theta$, we get
\begin{equation*}
  (\partial_t \tilde\theta)(t_*,x_*) = - \tilde\theta(t_*,x_*) - (v\cdot\nabla \tilde\theta)(t_*,x_*) + \Delta \tilde\theta(t_*,x_*).
\end{equation*}
Due to that $\tilde\theta$ attains the infimum at $(t_*,x_*)$, it yields that $(\nabla \tilde\theta)(t_*,x_*)=0$ and $(\Delta \tilde\theta)(t_*,x_*)\geq 0$,
thus we find
\begin{equation*}
  (\partial_t \tilde\theta)(t_*,x_*) \geq - \tilde\theta(t_*,x_*)=\lambda.
\end{equation*}
But this clearly contradicts with the fact that $(t_*,x_*)$ is the infimum point of $\tilde\theta$, hence the nonnegativity of $\theta$ for every $t\in [0,T]$
is followed. Note that in the above proof  the smoothness of $\theta$ is required.
So this part work for smooth approximation of solutions (see Section \ref{subsec:exe}). Passage to the limit saves the nonnegativity of the temperature.

Next, we show $\theta\in L^\infty([0,T]; L^1(\R^2))$.
Let $\varphi\in \mathcal{D}(\R^2)$ be a test function satisfying $\mathrm{supp}\,\varphi \subset B_1(0)$, $\varphi \equiv 1$ on $B_{1/2}(0)$, and $0\leq\varphi \leq 1$.
Set $\varphi_R:=\varphi(\frac{\cdot}{R})$ for every $R>0$. Multiplying both sides of the equation of $\theta$ with $\varphi_R$ and integrating on the spatial variable,
we obtain
\begin{equation*}
  \frac{\dd }{\dd t}\int_{\R^2} \theta(t) \varphi_R \dd x + \int_{\R^2} v\cdot\nabla \theta\, \varphi_R \dd x - \int_{\R^2} \Delta\theta\, \varphi_R \dd x = \int_{\R^2} \mu\, \varphi_R \dd x .
\end{equation*}
By viewing the measure $\mu(t)$ as the dual space of $C_0(\R^2)$, we deduce that
\begin{equation*}
  \int_{\R^2} \mu(t,x)\, \varphi_R(x) \dd x = (\mu(t),\varphi_R) \leq \|\mu(t)\|_{\mathcal{M}}\leq \|\mu_0\|_{\mathcal{M}}.
\end{equation*}
Thus integrating on the time interval $[0,t]$ ($t\in [0,T]$) and using integration by parts, we find
\begin{equation*}
  \int_{\R^2} \theta(t) \varphi_R \dd x \leq  \int_{\R^2} \theta_0 \varphi_R \dd x + \frac{1}{R} \int_0^T\int_{\R^2} |v\theta| \Big|\nabla \varphi(\frac{x}{R})\Big| \dd x \dd t
  + \frac{1}{R} \int_0^T\int_{\R^2} |\nabla\theta| \Big|\nabla \varphi(\frac{x}{R})\Big| \dd x \dd t + T \|\mu_0\|_{\mathcal{M}} .
\end{equation*}
Since we assume $\theta,v$ are smooth functions which guarantees that $v,\theta,\nabla\theta \in L^2([0,T]\times \R^2)$, by passing $R$ to $\infty$, it yields that
\begin{equation*}
  \int_{\R^2} \theta(t,x) \dd x \leq \int_{\R^2} \theta_0(x) \dd x
  + T\|\mu_0\|_{\mathcal{M}}.
\end{equation*}
Hence the desired inequality \eqref{theL1es} is followed from the nonnegativity of $\theta(t)$.
\end{proof}

\begin{proposition}\label{prop:ape3}
  Let $\mu_0\in \mathcal{M}_+(\R^2)$ with $\mathrm{supp}\,\mu_0\subset B_{R_0}(0)$ for some $R_0>0$. For each $\sigma\in ]0,2[$, let $\theta_0 \in L^1\cap B^{2-\sigma}_{\frac{4}{4-\sigma},\infty}(\R^2)$ with $\theta_0\geq 0$,
and $v_0\in H^1(\R^2)$ be a divergence-free vector field with initial vorticity $\omega_0:=\partial_1 v_{2,0}-\partial_2 v_{1,0}\in B^{3-\sigma}_{\frac{4}{4-\sigma},\infty}(\R^2)$.
Let $T>0$ be any given, and assume that $(\mu, \theta,v)$ are smooth functions on $\R^2\times [0,T]$ solving the system \eqref{BoussEq}.
Then we have
\begin{equation}\label{eq:ape1}
  \|\theta\|_{L^\infty_T (B^{2-\sigma}_{\frac{4}{4-\sigma},\infty}(\R^2))} + \|v\|_{L^\infty_T(H^1(\R^2))} + \|v\|_{L^2_T (H^2(\R^2))} \leq C e^{\exp(C(1+T)^8)},
\end{equation}
and
\begin{equation}\label{eq:ape2}
  \|\nabla v\|_{L^\infty_T (B^{3-\sigma}_{\frac{4}{4-\sigma},\infty}(\R^2))} + \|v\|_{L^\infty_T(W^{1,\infty}(\R^2))} + \|(\nabla p, \partial_t v)\|_{L^\infty_T (B^{2-\sigma}_{\frac{4}{4-\sigma},\infty}(\R^2))}
  \leq C e^{\exp(C(1+T)^8)},
\end{equation}
where $C>0$ depends only on $\sigma$ and the norms of initial data $(\mu_0,\theta_0,v_0)$.
\end{proposition}

\begin{proof}[Proof of Proposition \ref{prop:ape3}]
  We first consider the energy type estimates of $v$.
By taking the scalar product of the equation of velocity field $v$ with $v$ itself, we get
\begin{equation*}
\begin{split}
  \frac{1}{2}\frac{\dd }{\dd t}\|v(t)\|_{L^2}^2 + \|\nabla v(t)\|_{L^2}^2 \leq \left|\int_{\R^2} \theta\, v_2(t,x) \dd x \right|
  \leq \|\theta(t)\|_{L^1(\R^2)} \|v(t)\|_{L^\infty(\R^2)}.
\end{split}
\end{equation*}
By using $L^1$-estimate \eqref{theL1es} and the interpolation inequality, we infer that
\begin{equation}\label{v-eneEs1}
  \frac{1}{2}\frac{\dd }{\dd t}\|v(t)\|_{L^2}^2 + \|\nabla v(t)\|_{L^2}^2 \leq C(1+t) \|v(t)\|_{L^2(\R^2)}^{1/2} \|\nabla^2 v\|_{L^2(\R^2)}^{1/2},
\end{equation}
where $C>0$ depends on the norms of initial data $\|\mu_0\|_{\mathcal{M}}$ and $\|\theta_0\|_{L^1(\R^2)}$.
We then consider the equation of vorticity $\omega:= \mathrm{curl} \,v = \partial_1 v_2 - \partial_2 v_1$, which reads as
\begin{equation}\label{eq:omeg}
  \partial_t \omega + v\cdot\nabla \omega - \Delta \omega = \partial_1 \theta.
\end{equation}
By taking the inner product of the above equation with $\omega$, and using the integration by parts, we derive
\begin{equation*}
  \frac{1}{2} \frac{\dd }{\dd t} \|\omega(t)\|_{L^2}^2 + \|\nabla \omega(t)\|_{L^2}^2 \leq \left|\int_{\R^2} \theta\, \partial_1\omega(t,x)\dd x \right| \leq \|\theta(t)\|_{L^2} \|\nabla \omega(t)\|_{L^2}.
\end{equation*}
Young's inequality directly leads to
\begin{equation}\label{v-eneEs2}
  \frac{\dd }{\dd t}\|\omega(t)\|_{L^2(\R^2)}^2 + \|\nabla \omega(t)\|_{L^2(\R^2)}^2 \leq \|\theta(t)\|_{L^2(\R^2)}^2.
\end{equation}
Noting that $\|\nabla^2 v\|_{L^2} \leq \|\nabla \omega\|_{L^2}$ (from formula $v=\nabla^\perp(-\Delta)^{-1}\omega$) and $ab\leq \epsilon a^4 + C_\epsilon b^{4/3} $ for any $a,b,\epsilon>0$, we combine \eqref{v-eneEs2} with \eqref{v-eneEs1} to get
\begin{equation}\label{v-eneEs3}
  \frac{\dd }{\dd t}\left(\|v(t)\|_{L^2}^2 + \|\omega(t)\|_{L^2}^2\right) + \|\nabla v(t)\|_{L^2}^2 + \frac{1}{2}\|\nabla \omega(t)\|_{L^2}^2 \leq \|\theta(t)\|_{L^2(\R^2)}^2 + C(1+t)^{4/3} \|v(t)\|_{L^2(\R^2)}^{2/3}.
\end{equation}

In order to control the norm $\|\theta(t)\|_{L^2(\R^2)}$, we next consider the equation of $\theta$. Observing that $\mu(t)\in \mathcal{M}(\R^2)= (C_0(\R^2))^*$ and also
$$
B^\sigma_{\frac{4}{\sigma},2}(\R^2)\hookrightarrow B^{\sigma/2}_{\infty,2}(\R^2) \hookrightarrow B^0_{\infty,1}(\R^2)\hookrightarrow C_0(\R^2),
$$
it leads to that $\mu(t)\in (B^\sigma_{\frac{4}{\sigma},2}(\R^2))^* = B^{-\sigma}_{\frac{4}{4-\sigma},2}(\R^2)$ for every $t\in [0,T]$ (for the dual spaces of Besov spaces, one can see e.g. \cite[Proposition 1.3.5]{Danc}).
Note that this choice of the space seems to be not optimal, however this leads to lower power index of integrability of the ground space. Thanks to this choice the estimation closes.

Applying Lemma \ref{lem:heatEq} to the equation $\partial_t \theta -\Delta \theta = - v\cdot\nabla \theta + \mu$, we infer that for every $t\in [0,T]$,
\begin{equation*}
  \|\theta\|_{L^\infty([0,t]; B^{2-\sigma}_{\frac{4}{4-\sigma},\infty}(\R^2))} \leq C_0\Big(\|\theta_0\|_{B^{2-\sigma}_{\frac{4}{4-\sigma},\infty}} + (1+t)\|\mu\|_{L^\infty_t(B^{-\sigma}_{\frac{4}{4-\sigma},\infty})}
  + (1+t) \|v\cdot\nabla \theta\|_{L^\infty_t (B^{-\sigma}_{\frac{4}{4-\sigma},\infty})}\Big).
\end{equation*}
Thanks to estimate \eqref{mu-es} and Lemma \ref{lem:prodEs}, we get
\begin{equation}\label{eq:theEs1}
\begin{split}
  \|\theta\|_{L^\infty_t(B^{2-\sigma}_{\frac{4}{4-\sigma},\infty})} \leq & \, C_0\big(\|\theta_0\|_{B^{2-\sigma}_{\frac{4}{4-\sigma},\infty}} + (1+t)\|\mu_0\|_{\mathcal{M}}\big) \\
  & + C_0 (1+t) \|(v,\omega)\|_{L^\infty_t (L^2)}  \Big(\sup_{k\geq -1} 2^{k(1-\sigma)}\sqrt{2+k} \|\Delta_k\theta\|_{L^\infty_t (L^{\frac{4}{4-\sigma}})}\Big),
\end{split}
\end{equation}
where the usual abbreviation $\|(v,\omega)\|_{L^\infty_t (L^2)} := \|v\|_{L^\infty_t (L^2)} + \|\omega\|_{L^\infty_t (L^2)}$ has been adopted.

We first derive a rough estimate of $\|\theta\|_{L^\infty_t(B^{2-\sigma}_{\frac{4}{4-\sigma},\infty})}$ in terms of $\|(v,\omega)\|_{L^\infty_t (L^2)} $.
By using the interpolation inequality and Young's inequality, it follows that
\begin{align}\label{eq:theEst}
  &\|\theta\|_{L^\infty_t(B^{2-\sigma}_{\frac{4}{4-\sigma},\infty})}
  \leq C_0 \Big(\|\theta_0\|_{B^{2-\sigma}_{\frac{4}{4-\sigma},\infty}} + (1+t) \|\mu_0\|_{\mathcal{M}} +  (1+t) \|(v,\omega)\|_{L^\infty_t (L^2)} \|\theta\|_{L^\infty_t (B^{1-\frac{\sigma}{2}}_{\frac{4}{4-\sigma},\infty})}\Big) \nonumber\\
  & \leq C_0\Big(\|\theta_0\|_{B^{2-\sigma}_{\frac{4}{4-\sigma},\infty}} + (1+t) \|\mu_0\|_{\mathcal{M}}
  + (1+t) \|(v,\omega)\|_{L^\infty_t (L^2)}\|\theta\|_{L^\infty_t (L^1)}^{\frac{2-\sigma}{4-\sigma}} \|\theta\|_{L^\infty_t (B^{2-\sigma}_{\frac{4}{4-\sigma},\infty})}^{\frac{2}{4-\sigma}}\Big)\nonumber \\
  & \leq  C_0\Big(\|\theta_0\|_{B^{2-\sigma}_{\frac{4}{4-\sigma},\infty}} + (1+t) \|\mu_0\|_{\mathcal{M}}
  + (1+t)^{\frac{4-\sigma}{2-\sigma}} \|(v,\omega)\|_{L^\infty_t (L^2)}^{\frac{4-\sigma}{2-\sigma}} \|\theta\|_{L^\infty_t (L^1)}\Big)+ \frac{1}{2}\|\theta\|_{L^\infty_t (B^{2-\sigma}_{\frac{4}{4-\sigma},\infty})},
%  & \leq C(1+t)^{3-\frac{\sigma}{2}} (1+ \|v\|_{L^\infty_t (B^0_{\infty,1})}^{2-\sigma/2}) + \frac{1}{2}\|\theta\|_{L^\infty_t (B^{2-\sigma}_{\frac{4}{4-\sigma},\infty})},
\end{align}
thus using estimate \eqref{theL1es} leads to
\begin{equation}\label{theEsBesov}
  \|\theta\|_{L^\infty_t(B^{2-\sigma}_{\frac{4}{4-\sigma},\infty})} \leq C(1+t)^{\frac{6-2\sigma}{2-\sigma}} \Big(1+ \|(v,\omega)\|_{L^\infty_t (L^2)}^{\frac{4-\sigma}{2-\sigma}}\Big),
\end{equation}
with $C$ depending on the norms $\|\theta_0\|_{L^1\cap B^{2-\sigma}_{\frac{4}{4-\sigma},\infty}}$ and $\|\mu_0\|_{\mathcal{M}}$.

Then we show a more refined estimate of \eqref{theEsBesov} by slightly reducing the power index of $\|(v,\omega)\|_{L^\infty_t L^2}$.
Through applying the interpolation inequality \eqref{eq:intp}, $L^1$-estimate \eqref{theL1es} and the fact that the function $z\mapsto z^{\frac{2}{4-\sigma}} \sqrt{\log(e+ \frac{1}{z})}$
is increasing on $]0,\infty[$, we find
\begin{align}\label{eq:theEs2}
  &\|\theta\|_{L^\infty_t(B^{2-\sigma}_{\frac{4}{4-\sigma},\infty})}
  \leq\, C_0\Big(\|\theta_0\|_{B^{2-\sigma}_{\frac{4}{4-\sigma},\infty}} + (1+t)\|\mu_0\|_{\mathcal{M}}\Big) \nonumber\\
  & + C (1+t) \|(v,\omega)\|_{L^\infty_t (L^2)}  \bigg(\|\theta\|_{L^\infty_t (L^1)}^{\frac{2}{4-\sigma}} \|\theta\|_{L^\infty_t (B^{2-\sigma}_{\frac{4}{4-\sigma},\infty})}^{\frac{2-\sigma}{4-\sigma}}
  \sqrt{\log\Big(e + \frac{\|\theta\|_{L^\infty_t (B^{2-\sigma}_{4/(4-\sigma),\infty})}}{\|\theta\|_{L^\infty_t(L^1)}}\Big)} + \|\theta\|_{L^\infty_t (L^1)}\bigg) \nonumber\\
  & \leq C (1+t) + C (1+t)^2 \|(v,\omega)\|_{L^\infty_t (L^2)}
  \bigg(\|\theta\|_{L^\infty_t (B^{2-\sigma}_{\frac{4}{4-\sigma},\infty})}^{\frac{2-\sigma}{4-\sigma}}
  \sqrt{\log\Big(e + \|\theta\|_{L^\infty_t (B^{2-\sigma}_{\frac{4}{4-\sigma},\infty})}\Big)} + 1 \bigg),
\end{align}
where $C$ depends on the norms of initial data.
By virtue of estimate \eqref{theEsBesov}, we also see that
\begin{align*}
  \sqrt{\log\Big(e + \|\theta\|_{L^\infty_t (B^{2-\sigma}_{\frac{4}{4-\sigma},\infty})}\Big)}
%  & \leq \sqrt{\log \Big( e + C(1+t)^{\frac{6-2\sigma}{2-\sigma}} \Big(1+ \|(v,\omega)\|_{L^\infty_t (L^2)}^{\frac{4-\sigma}{2-\sigma}}\Big)\Big)} \\
  & \leq \sqrt{\log \Big( \big(e+  C(1+t)^{\frac{6-2\sigma}{2-\sigma}}\big)
  \Big(e + \|(v,\omega)\|_{L^\infty_t (L^2)}^{\frac{4-\sigma}{2-\sigma}}\Big)\Big)} \\
  & \leq \sqrt{1+\log\Big(e+  C(1+t)^{\frac{6-2\sigma}{2-\sigma}} \Big) }
  \sqrt{\log\Big(e + \|(v,\omega)\|_{L^\infty_t (L^2)}^{\frac{4-\sigma}{2-\sigma}}\Big)} \\
  & \leq C (1+t) \sqrt{\log\Big(e + \|(v,\omega)\|_{L^\infty_t (L^2)}^2\Big)},
\end{align*}
thus inserting this inequality into \eqref{eq:theEs2} leads to that
\begin{equation*}
  \|\theta\|_{L^\infty_t(B^{2-\sigma}_{\frac{4}{4-\sigma},\infty})}
  \leq C (1+t) + C (1+t)^3 \|(v,\omega)\|_{L^\infty_t (L^2)}
  \bigg(\sqrt{\log\Big(e + \|(v,\omega)\|_{L^\infty_t (L^2)}^2\Big)}\|\theta\|_{L^\infty_t (B^{2-\sigma}_{\frac{4}{4-\sigma},\infty})}^{\frac{2-\sigma}{4-\sigma}} +1\bigg).
\end{equation*}
By arguing as \eqref{eq:theEst} and \eqref{theEsBesov}, we obtain
\begin{equation}\label{theEsBesov2}
  \|\theta\|_{L^\infty_t(B^{2-\sigma}_{\frac{4}{4-\sigma},\infty})} \leq
  C (1+t)^{\frac{3(4-\sigma)}{2}}\Big(1+ \|(v,\omega)\|_{L^\infty_t (L^2)}^{\frac{4-\sigma}{2}}\Big) \Big(\log\Big(e + \|(v,\omega)\|_{L^\infty_t (L^2)}^2\Big)\Big)^{\frac{4-\sigma}{4}}.
\end{equation}

Now we go back to inequality \eqref{v-eneEs3}. By using the interpolation inequality, estimates \eqref{theL1es} and \eqref{theEsBesov2}, we deduce that
\begin{equation*}
\begin{split}
  \frac{\dd }{\dd t} \|(v,\omega)(t)\|_{L^2}^2 + \frac{1}{2}\|(\nabla v,\nabla \omega)(t)\|_{L^2}^2
  & \leq C \|\theta\|_{L^\infty_t L^1}^{\frac{2(2-\sigma)}{4-\sigma}} \|\theta\|_{L^\infty_t B^{2-\sigma}_{\frac{4}{4-\sigma},\infty}}^{\frac{4}{4-\sigma}} + C(1+t)^{\frac{4}{3}} \|v(t)\|_{L^2(\R^2)}^{\frac{2}{3}} \\
  & \leq C (1+ t)^7 \bigg(1+ \|(v,\omega)\|_{L^\infty_t (L^2)}^2 \log\Big(e + \|(v,\omega)\|_{L^\infty_t (L^2)}^2\Big)\bigg).
\end{split}
\end{equation*}
Integrating on the time variable yields that for every $t\in [0,T]$,
\begin{equation}\label{eq:vEs-key0}
\begin{split}
   & \|(v,\omega)\|_{L^\infty_t(L^2)}^2 + \|(\nabla v,\nabla \omega)\|_{L^2_t (L^2)}^2 \\
   \leq & \,C_0\|(v_0,\omega_0)\|_{L^2}^2 + C (1+ t)^8
   + C \int_0^t (1+ \tau)^7 \|(v,\omega)\|_{L^\infty_\tau (L^2)}^2 \log\Big(e + \|(v,\omega)\|_{L^\infty_\tau (L^2)}^2\Big) \dd \tau.
\end{split}
\end{equation}
Gr\"onwall's inequality guarantees that
\begin{equation}\label{eq:vEs-key}
  \|(v,\omega)\|_{L^\infty_T(L^2)}^2 + \|(\nabla v,\nabla \omega)\|_{L^2_T (L^2)}^2
  \leq C e^{\exp(C(1+T)^8)},
\end{equation}
where $C$ depends on the norms $\|v_0\|_{H^1}$, $\|\theta_0\|_{L^1\cap B^{2-\sigma}_{\frac{4}{4-\sigma},\infty}}$ and $\|\mu_0\|_{\mathcal{M}}$.
Plugging the above estimate into \eqref{theEsBesov} leads to
\begin{equation}\label{theEsBesov3}
  \|\theta\|_{L^\infty_t(B^{2-\sigma}_{\frac{4}{4-\sigma},\infty})} \leq C e^{\exp(C(1+T)^8)},
\end{equation}
which combined with \eqref{eq:vEs-key} and the facts $\|\nabla v\|_{L^2}\leq \|\omega\|_{L^2}$ and $\|\nabla^2 v\|_{L^2}\leq \|\nabla \omega\|_{L^2}$
implies the desired estimate \eqref{eq:ape1}.

Next we turn to the proof of estimate \eqref{eq:ape2}. By viewing the equation of $\omega$ \eqref{eq:omeg} as a heat equation with forcing, we use estimates \eqref{eq:heat-ape} and \eqref{eq:prodEs2} to get
\begin{align*}
  \|\omega\|_{L^\infty_T (B^{3-\sigma}_{\frac{4}{4-\sigma},\infty})} \leq\, C_0 \Big(\|\omega_0\|_{B^{3-\sigma}_{\frac{4}{4-\sigma},\infty}(\R^2)} + (1+T) \|\partial_1\theta\|_{L^\infty_T (B^{1-\sigma}_{\frac{4}{4-\sigma},\infty})}
  + (1+T) \|v\cdot \nabla \omega\|_{L^\infty_T (B^{1-\sigma}_{\frac{4}{4-\sigma},\infty})}\Big) & \\
  \leq\, C_0 \Big(\|\omega_0\|_{B^{3-\sigma}_{\frac{4}{4-\sigma},\infty}} + (1+T) \|\theta\|_{L^\infty_T (B^{2-\sigma}_{\frac{4}{4-\sigma},\infty})}
  + (1+T) \|v\|_{L^\infty_T (B^{1-\sigma}_{\frac{8}{4-\sigma},\infty})} \|\omega\|_{L^\infty_T (B^{2-\sigma}_{\frac{8}{4-\sigma},\infty})} \Big) .&
\end{align*}
In view of estimates \eqref{eq:vEs-key}-\eqref{theEsBesov3}, the continuous embedding $H^1(\R^2)\hookrightarrow B^{1-\sigma}_{\frac{8}{4-\sigma},\infty}(\R^2)$, the interpolation inequality $\|g\|_{B^{2-\sigma}_{\frac{8}{4-\sigma},\infty}(\R^2)}\leq C \|g\|_{L^2(\R^2)}^{\frac{\sigma}{8-2\sigma}}
\|g\|_{B^{3-\sigma}_{\frac{4}{4-\sigma},\infty}(\R^2)}^{\frac{8-3\sigma}{8-2\sigma}}$ and Young's inequality, we infer that
\begin{align*}
  \|\omega\|_{L^\infty_T (B^{3-\sigma}_{\frac{4}{4-\sigma},\infty}(\R^2))} & \leq\, C e^{\exp(C(1+T)^8)} +C_0 (1+T) \|v\|_{L^\infty_T (H^1)} \|\omega\|_{L^\infty_T (L^2)}^{\frac{\sigma}{8-2\sigma}}
  \|\omega\|_{L^\infty_T (B^{3-\sigma}_{\frac{4}{4-\sigma},\infty}(\R^2))}^{\frac{8-3\sigma}{8-2\sigma}} \\
  & \leq C e^{\exp(C(1+T)^8)} + \big(C_0(1+T) \|v\|_{L^\infty_T (H^1)} \big)^{\frac{8-2\sigma}{\sigma}}\|\omega\|_{L^\infty_T (L^2)} + \frac{1}{2}\|\omega\|_{L^\infty_T (B^{3-\sigma}_{\frac{4}{4-\sigma},\infty}(\R^2))} \\
  & \leq C e^{\exp(C(1+T)^8)} + \frac{1}{2}\|\omega\|_{L^\infty_T (B^{3-\sigma}_{\frac{4}{4-\sigma},\infty}(\R^2))},
\end{align*}
thus the Calder\'on-Zygmund theorem implies
\begin{equation}\label{eq:omeEsBesov}
  \|\nabla v\|_{L^\infty_T (B^{3-\sigma}_{\frac{4}{4-\sigma},\infty}(\R^2))}\leq C \|\omega\|_{L^\infty_T (B^{3-\sigma}_{\frac{4}{4-\sigma},\infty}(\R^2))} \leq C e^{\exp(C(1+T)^8)},
\end{equation}
where $C$ depends on the norms $\|v_0\|_{H^1}$, $\|\omega_0\|_{B^{3-\sigma}_{\frac{4}{4-\sigma},\infty}}$, $\|\theta_0\|_{L^1\cap B^{2-\sigma}_{\frac{4}{4-\sigma},\infty}}$ and $\|\mu_0\|_{\mathcal{M}}$.
Besides, by virtue of the high-low frequency decomposition and Bernstein's inequality, we have
\begin{equation}\label{eq:vLipEs}
\begin{split}
  \|v\|_{L^\infty_T (W^{1,\infty}(\R^2))} & \leq C_0 \|\Delta_{-1}v\|_{L^\infty(\R^2)} + C_0 \sum_{q\in \N}  \|\Delta_q \nabla v\|_{L^\infty(\R^2)} \\
  & \leq C_0 \|v\|_{L^\infty_T (L^2(\R^2))} + C_0 \sum_{q\in \N} 2^{q(\frac{\sigma}{2}-1)} 2^{q(3-\sigma)}\|\Delta_q \nabla v\|_{L^\infty_T (L^{\frac{4}{4-\sigma}}(\R^2))} \\
  & \leq C_0 \|v\|_{L^\infty_T (L^2(\R^2))} + C \Big(\sum_{q\in \N} 2^{q(\frac{\sigma}{2}-1)}\Big) \|\omega\|_{L^\infty_T (B^{3-\sigma}_{\frac{4}{4-\sigma},\infty}(\R^2))} \\
  & \leq C e^{\exp(C(1+T)^8)},
\end{split}
\end{equation}
as desired. Here we see the straightforward proof of the continuous embedding $B^{3-\sigma}_{\frac{4}{4-\sigma},\infty}(\R^2)
\hookrightarrow L^\infty(\R^2)$
for $\sigma \in ]0,2[$ in order to keep $\frac{\sigma}{2} - 1 <0$.

By using the third equation of system \eqref{BoussEq}, the divergence-free condition of $v$ and the above \textit{a priori} estimates on $v,\theta$, we see that
\begin{equation*}
  \nabla p = \nabla (-\Delta)^{-1} \divg(v\cdot\nabla v) - \nabla \partial_1 (-\Delta)^{-1}\theta,
\end{equation*}
which combined with the Calder\'on-Zygmund theorem and inequality \eqref{eq:prodEs3} leads to
\begin{align}
  \|\nabla p\|_{L^\infty_T(B^{2-\sigma}_{\frac{4}{4-\sigma},\infty}(\R^2))} \leq\, & C \|v\cdot\nabla v\|_{L^\infty_T(B^{2-\sigma}_{\frac{4}{4-\sigma},\infty})}
  + C \|\theta\|_{L^\infty_T (B^{2-\sigma}_{\frac{4}{4-\sigma},\infty})} \nonumber \\
  \leq\, & C \|v\|_{L^\infty_T (W^{1,\infty})} \Big(\|\nabla v\|_{L^\infty_T (B^{2-\sigma}_{\frac{4}{4-\sigma},\infty})} + \|v\|_{L^\infty_T (B^{2-\sigma}_{\frac{4}{4-\sigma},\infty})} \Big)
  + C \|\theta\|_{L^\infty_T (B^{2-\sigma}_{\frac{4}{4-\sigma},\infty})} \nonumber \\
  \leq\, & C e^{\exp(C(1+T)^8)}.
\end{align}
Furthermore, from the third equation of system \eqref{BoussEq}, we derive
\begin{align}
  \|\partial_t v\|_{L^\infty_T (B^{2-\sigma}_{\frac{4}{4-\sigma},\infty}(\R^2))} & \leq \|v\cdot\nabla v \|_{L^\infty_T (B^{2-\sigma}_{\frac{4}{4-\sigma},\infty})}
  + \|\Delta v\|_{L^\infty_T (B^{2-\sigma}_{\frac{4}{4-\sigma},\infty})} + \|\nabla p\|_{L^\infty_T (B^{2-\sigma}_{\frac{4}{4-\sigma},\infty})}
  +  \|\theta\|_{L^\infty_T (B^{2-\sigma}_{\frac{4}{4-\sigma},\infty})} \nonumber \\
  & \leq C e^{\exp(C(1+T)^8)} + C \|\nabla v\|_{L^\infty_T (B^{3-\sigma}_{\frac{4}{4-\sigma},\infty})} \leq C e^{\exp(C(1+T)^8)}.
\end{align}

\end{proof}

\subsection{Global existence}\label{subsec:exe}
The issue of existence for our system is not immediate since $\mu=\mu(x,t)$ is
just a measure. In order to construct a suitable approximation we
consider the system with smooth initial data. One can start with
an initial sequence
\begin{equation}\label{g1}
 \mu^{(n)}|_{t=0}= \phi^n*\mu_0,\quad \theta^{(n)}|_{t=0}= \phi^n*\theta_0,\quad v^{(n)}|_{t=0}=\phi^n*v_0,
 \end{equation}
where
$n\in\N^+$ is an approximation parameter ($n\to \infty$ in the end) and $\phi^n(x) = n^2 \phi(nx)$ with $\phi\in \mathcal{D}(\R^2)$ a standard mollifier function.

To show the existence of system \eqref{BoussEq} with such initial data \eqref{g1}, we
want to use a standard approach via Galerkin method.
An approximation we build on the following spaces:

\smallskip

$\ast$ \ $H^2(\R^2)$ for the velocity field in the divergence-free subset

\smallskip
and
\smallskip

$\ast$ \ $H^1(\R^2)$ for the temperature.

\smallskip

%The function $\mu$ we will compute form the transported equation.
%
In short, $v^{(n),N}$ and $\theta^{(n),N}$ are approximations based on the
$N$-dimensional restriction of $H^2$ and $H^1$ spaces.
We have
\begin{equation}
 v^{(n),N}=\sum_{k=1}^N V^{(n),N}_k(t)w_k(x), \qquad
 \theta^{(n),N}=\sum_{k=1}^N \Theta_k^{(n),N}(t) g_k(x).
\end{equation}
Vectors $w_k$ and $g_l$ are the based vectors of $H^2(\R^2;\R^2)$
of the divergence-free subspace and $H^1(\R^2;\R)$, respectively.
The sought functions $V^{(n),N}_k(t)$ and $\Theta_k^{(n),N}(t)$ are derived by solving of
the following ODEs
\begin{equation}\label{appSys-ODE}
 \begin{array}{l}
  (\theta^{(n),N}_t,\psi^N) + (v^{(n),N} \cdot\nabla \theta^{(n),N}, \psi^N)
  +(\nabla \theta^{(n),N},\nabla \psi^N)=(\mu^{(n),N},\psi^N), \\[10pt]
  (v^{(n),N}_t,\Psi^N)+ (v^{(n),N} \cdot\nabla v^{(n),N},\Psi^N)+
  (\nabla v^{(n),N},\nabla \Psi^N)=(\theta^{(n),N} e_2,\Psi^N), \\[10pt]
  V^{(n),N}_k(0) = (\phi^n*v_0, w_k),\quad \Theta^{(n),N}(0)= (\phi^n *\theta_0, g_k),\quad k=1,\cdots,N,
 \end{array}
\end{equation}
for all $\Psi^N \in {\rm span}\{w_1,...,w_N\} \subset H^2(\R^2;\R^2)$
with $\divg \Psi^N=0$, and $\psi^N \in {\rm span}\{g_1,...,g_N\} \subset
H^1(\R^2;\R)$.
And $\mu^{(n),N}$ is the classical solution to the transport equation
\begin{equation}
 \mu^{(n),N}_t + v^{(n),N} \cdot \nabla \mu^{(n),N}=0,\quad \mu^{(n),N}|_{t=0}=  \phi^n*\mu_0.
\end{equation}
%with smoothed out initial condition.

The local in time existence for the system is clear, and in order to pass to the limit with $N$ we need just the \textit{a priori} estimate in suitable
energy norms independent of $N$, which of course depends on $T$ but never blows up for any finite $T$.

Note that the condition $\divg v^{(n),N}=0$ leads to the following bound uniformly in $N$:
\begin{equation}
 \mu^{(n),N} \in L^\infty(0,T;L^2(\R^2)), \mbox{ \ \ indeed \ } \mu^{(n),N} \in L^\infty(0,T;L^1\cap L^\infty(\R^2)),
\end{equation}
since by definition $\|\phi_n*\mu_0 \|_{L^1(\R^2)}\leq C_0 \|\mu_0\|_{\mathcal{M}(\R^2)}$ (uniformly in $n$) and from Young's inequality $\phi_n*\mu_0 \in L^1\cap L^\infty(\R^2)$ for every $n\in\N^+$.
Hence testing the first equation by $\theta^{(n),N}$ in \eqref{appSys-ODE} we get
\begin{equation}
\theta^{(n),N} \in L^\infty(0,T; L^2(\R^2)) \cap L^2(0,T;H^1(\R^2)),\quad \textrm{uniformly in}\;N.
\end{equation}
Then testing the second equation in \eqref{appSys-ODE} by $\Delta v^{(n),N}$, and using the structure of the two dimensional Navier-Stokes Equations we get
\begin{equation}
 v^{(n),N} \in L^\infty(0,T;H^1(\R^2)) \cap L^2(0,T;H^2(\R^2)) \cap  H^1(0,T;L^2(\R^2)),\quad \textrm{uniformly in}\;N.
\end{equation}
The above information guarantees us strong convergence of $(\mu^{(n),N},\theta^{(n),N},v^{(n),N})$ locally in space as $N\rightarrow \infty$.
Hence there is no problem to pass to the limit $N\rightarrow \infty$ and we get the solution to the system (\ref{BoussEq}) with initial data given by (\ref{g1}), i.e.:
\begin{equation}\label{appBEq}
\begin{array}{ll}
  \partial_t \mu^{(n)} + v^{(n)}\cdot \nabla \mu^{(n)} =0,
  & \mbox{in \ } \R^2 \times ]0,T],\\
  \partial_t \theta^{(n)} + v^{(n)}\cdot\nabla \theta^{(n)} -\Delta \theta^{(n)} = \mu^{(n)},
  & \mbox{in \ } \R^2 \times ]0,T], \\
  \partial_t v^{(n)} + v^{(n)} \cdot \nabla v^{(n)} -\Delta v^{(n)} + \nabla p^{(n)} = \theta^{(n)},
  & \mbox{in \ } \R^2 \times ]0,T], \\
  \mathrm{div}\, v^{(n)}=0, & \mbox{in \ } \R^2 \times ]0,T], \\
  \mu^{(n)}|_{t=0}= \phi^n*\mu_0,\quad \theta^{(n)}|_{t=0}= \phi^n*\theta_0,\quad v^{(n)}|_{t=0}=\phi^n*v_0.
\end{array}
\end{equation}
Using the standard bootstap method (here we use just the simple structure of quasi-linear systems) we obtain that for every $n\in\N^+$ and for any $1<q,p<\infty$,
\begin{equation}\label{app-higReg}
 \begin{array}{l}
  \mu^{(n)} \in L^\infty(0,T;L^1\cap L^\infty(\R^2)),\\
  \theta^{(n)} \in L^q(0,T;W^{2,p})\cap W^{1,q}(0,T;L^p(\R^2)),\\
  v^{(n)} \in L^q(0,T;W^{4,p}(\R^2) \cap W^{2,q}(0,T;L^p(\R^2)).
 \end{array}
\end{equation}
Higher regularity of the approximative sequence ensures that they satisfy the a priori estimates from subsection \ref{subsec:apE}. Hence we have
\begin{equation}\label{mu-apEs}
  \textrm{$\|\mu^{(n)}\|_{L^\infty(0,T; \mathcal{M}(\R^2))} \leq \|\mu_0\|_{\mathcal{M}(\R^2)}$,\quad $\mu^{(n)}\geq 0$,\quad and\quad $\mathrm{supp}\,\mu^{(n)}\subset B_{R_0 +C}(0)$.}
\end{equation}
and
\begin{equation}\label{the-v-apEs0}
  \|\theta^{(n)}\|_{L^\infty_T (L^1\cap B^{2-\sigma}_{\frac{4}{4-\sigma},\infty})}+
  \|v^{(n)}\|_{L^\infty_T (H^1\cap W^{1,\infty})}
  +
  \|\nabla v^{(n)}\|_{L^\infty_T (B^{3-\sigma}_{\frac{4}{4-\sigma},\infty})} + \|(\partial_t v^{(n)},\nabla p^{(n)})\|_{L^\infty_T (B^{2-\sigma}_{\frac{4}{4-\sigma},\infty})}\leq C,
\end{equation}
where $C$ is depending on $T$ and norms of initial data $(\mu_0,\theta_0,v_0)$ but independent of $n\in\N^+$.

Now we  analyse a possible limit of the sequence as $n\to \infty$. For $v^{(n)}$ and $\theta^{(n)}$, from \eqref{the-v-apEs0} and based on the standard compactness
argument for the Besov/Soblev spaces, we find a subsequence with strong (point-wise) convergence to some functions $v$ and $\theta$, more precisely, one has that for every $\varphi \in \mathcal{D}(\R^2)$,
\begin{equation}
\begin{split}
  &\varphi v^{(n)}\rightarrow \varphi v,\quad \textrm{in   }L^\infty(0,T; L^2\cap W^{1,\infty}(\R^2)), \\
  &\varphi \theta^{(n)} \rightarrow \varphi \theta,\quad \textrm{in   } L^\infty(0,T;L^2(\R^2)).
\end{split}
\end{equation}
%which is sufficient .
For $\mu^{(n)}$, we view it as a mapping from $[0,T]$ to the metric space $(\mathcal{M}_+,d)$, and we show that $\mu^{(n)}$ has a strong convergence by using the Arzela-Ascoli theorem.
The uniform boundedness and relative compactness of $\mu^{(n)}(t)$ are followed from \eqref{mu-apEs} and Proposition \ref{pro:mu-conv}, and for the equicontinuity property of $\mu^{(n)}(t)$,
we observe that for every $s_1,s_2\in [0,T]$ and every $\pi\in W^{1,\infty}(\R^2)$,
\begin{align}\label{equi}
  \left|\int_{\R^2} \big(\mu^{(n)}(s_2) -\mu^{(n)}(s_1)\big) \pi \dd x\right| & = \left|\int_{s_1}^{s_2}\int_{\R^2} v^{(n)} \mu^{(n)} \nabla \pi \dd x \dd t \right|  \leq C Lip(\pi) \|\mu_0\|_{\mathcal{M}} |s_2-s_1|,
\end{align}
so that
\begin{equation*}
  d(\mu^{(n)}(s_2),\mu^{(n)}(s_1)) \leq C |s_2-s_1|.
\end{equation*}
Hence the assumptions of Arzela-Ascoli theorem are satisfied and there exists $\mu\in L^\infty(0,T;\mathcal{M}_+)$ such that, up to a subsequence,
\begin{equation}\label{mu-converg}
  \mu^{(n)}(t) \to \mu(t) \mbox{ in }\mbox{$d$-topology uniformly in time.}
\end{equation}
The above procedure is viewed as standard in the transport theory, for details we refer e.g. for \cite{MP18}.

The information \eqref{mu-apEs} also implies $\mu\in L^\infty(0,T; \mathcal{M}_+(\R^2))$ and $\mathrm{supp}\,\mu\subset B_{R_0 +C}(0)$.
In addition, in view of definition \ref{d-topology}, the bound \eqref{mu-converg} and strong convergence of the velocity in $L^\infty(0,T;W^{1,\infty}_{\mathrm{loc}}(\R^2))$ guarantee that as $n\rightarrow \infty$ we have
\begin{equation}
 v^{(n)} \mu^{(n)} \to v\mu\quad  \mbox{ in  } \mathcal{D}'(\R^2 \times [0,T]).
\end{equation}
We thus have the existence.

%{\bf Piotr: the basic question is which point shall be explained better ??}

% For such system the existence follows immediately from the Schauder
% fixed point theorem applied to the map
% \begin{equation}
%  T(w)=v \mbox{ \ \ and \ \ } T: L^\infty(0,T;W^1_\infty(B_R))
%  \cap \{ \divg f=0\}
%  \to L^\infty(0,T;W^1_\infty(B_R)) \cap \{ \div f=0\}.
% \end{equation}
% % where $B_R=\{|x|<R\}$ and $v$ is given as a solution to the following problem
% \begin{equation}\label{appBEq}
% \begin{array}{ll}
%   \partial_t \mu + K_R w\cdot \nabla \mu =0
%   & \mbox{in \ } \R^2 \times (0,T),\\
%   \partial_t \theta + K_R w\cdot\nabla \theta -\Delta \theta = \mu
%   & \mbox{in \ } \R^2 \times (0,T), \\
%   \partial_t v + K_R w\cdot \nabla v -\Delta v + \nabla p = \theta
%   & \mbox{in \ } \R^2 \times (0,T), \\
%   \mathrm{div}\, v=0& \mbox{in \ } \R^2 \times (0,T), \\
% % %   \mu|_{t=0}= \phi^n*\mu_0,\quad \theta|_{t=0}= \phi^n*\theta_0,\quad v|_{t=0}=\phi^n*v_0,
% \end{array}
% \end{equation}
% % % % % % % % % % where $n\in\N^+$, $\phi^n(x) = n^2 \phi(nx)$ with $\phi\in C^\infty_c(\R^2)$ a standard mollifier function and function $K_R:\R^2 \to [0,1]$ is smooth and $K_R(x)=1$ for $|x|< R-1$ and $K_R=0$ for $|x|>R$.

\subsection{Uniqueness}\label{subsec:uniq}

Consider two solutions $(\mu_1,\theta_1,v_1,p_1)$ and $(\mu_2,\theta_2,v_2,p_2)$ to the Boussinesq type system \eqref{BoussEq}
starting from the same initial data $(\mu_0,\theta_0,v_0)$ as stated in Theorem \ref{thm:2Dgwp}.
%$\mu_0\in \mathcal{M}_+(\R^2,TV)$, $\theta_0\in L_1\cap  B^{2-\sigma}_{\frac{4}{4-\sigma},\infty}(\R^2)$ and $v_0\in H_1(\R^2)$,
%$\omega_0=\partial_1 v_{0,2}-\partial_2 v_{0,1}\in B^{3-\sigma}_{\frac{4}{4-\sigma},\infty}(\R^2)$.
According to Proposition \ref{prop:ape3}, we know that for $i=1,2$ and for any $T>0$ large,
\begin{equation}\label{the-v-apEs}
  \|\theta_i\|_{L^\infty_T (L^1\cap B^{2-\sigma}_{\frac{4}{4-\sigma},\infty})}<\infty,\quad \|v_i\|_{L^\infty_T (H^1\cap W^{1,\infty})}
  + \|\nabla v_i\|_{L^\infty_T (B^{3-\sigma}_{\frac{4}{4-\sigma},\infty})} + \|(\partial_t v_i,\nabla p_i)\|_{L^\infty_T (B^{2-\sigma}_{\frac{4}{4-\sigma},\infty})}<\infty.
\end{equation}
Denoting by $\bar{v}_i(t,y)= v_i(t, X_{v_i}(t,y))$ with $X_{v_i}(t,y)$ the particle-trajectory generated by $v_i$ (see \eqref{flow}), and letting $T'>0$ be small enough, we have
\begin{equation}\label{bar-v-Lip1}
\begin{split}
  \int_0^{T'} \|\nabla_y \bar{v}_i(t, y)\|_{L^\infty_y}\dd t & \leq \int_0^{T'} \|\nabla_x v_i(t,x)\|_{L^\infty_x} \|\nabla_y X_{v_i}(t,y)\|_{L^\infty_y} \dd t \\
  & \leq \int_0^{T'} \|\nabla_x v_i(t)\|_{L^\infty} e^{\int_0^t \|\nabla_x v_i(\tau)\|_{L^\infty}\dd \tau} \dd t \\
  & \leq T' \|\nabla v_i\|_{L^\infty_T L^\infty} e^{T' \|\nabla v_i\|_{L^\infty_T L^\infty}} \leq c_0,
\end{split}
\end{equation}
where $0<c_0\leq \frac{1}{2}$ is a fixed constant chosen later.
By adopting the notations introduced in subsection \ref{subsec:lag} and using \eqref{L-BoussEq2}, the system of $(\mu_i,\theta_i,v_i,p_i)$ ($i=1,2$) in the Lagrangian coordinates is written as
\begin{equation}\label{L-BEq3}
\begin{cases}
  \partial_t \bar{\theta}_i - \divg\left( A_{v_i} A_{v_i}^{\textrm{T}} \nabla_y \bar{\theta}_i\right)  = \mu_0, \\
  \partial_t \bar{v}_i - \divg\left( A_{v_i} A_{v_i}^{\textrm{T}} \nabla_y \bar{v}_i\right)  + A_{v_i}^{\textrm{T}}\nabla_y \overline{p}_i = \bar{\theta}_i \,e_2, \\
  \divg_y\left( A_{v_i} \bar{v}_i \right) =0, \\
  \bar{\theta}_i|_{t=0}= \theta_0,\quad \bar{v}_i|_{t=0}=v_0.
\end{cases}
\end{equation}
The choice of the Lagrangian coordinates setting removes the problem with uniqueness for measure force $\mu$. They are given explicitly as follows
\begin{equation}
 \mu_1(t,X_{v_1}(t,y))=\mu_2(t,X_{v_2}(t,y))=\mu_0(y).
\end{equation}

We see the difference equations of $\bar{\theta}_1-\bar{\theta}_2=: \delta \bar \theta$ and $\bar{v}_1 -\bar{v}_2=: \delta \bar{v}$ read as follows
\begin{equation}\label{L-BEq-del}
\begin{cases}
  \partial_t \delta \bar{\theta} -\divg\big( A_{v_1} A_{v_1}^{\textrm{T}} \nabla \delta \bar{\theta}\big)  =
  \divg\left( (A_{v_1} A_{v_1}^{\textrm{T}} -  A_{v_2}A_{v_2}^{\textrm{T}})\nabla  \bar{\theta}_2 \right), \\
  \partial_t \delta \bar{v} -\divg\big( A_{v_1} A_{v_1}^{\textrm{T}} \nabla \delta \bar{v}\big) + A_{v_1}^{\textrm{T}}\nabla \delta \overline{p} =
  (\delta \bar{\theta}) e_2 + \divg\left( (A_{v_1} A_{v_1}^{\textrm{T}} -
  A_{v_2}A_{v_2}^{\textrm{T}})\nabla  \bar{v}_2 \right)-(A_{v_1}^{\textrm{T}}-A_{v_2}^{\textrm{T}})\nabla  \overline{p}_2 , \\
  \divg \big( A_{v_1}\delta\bar{v}\big) = \divg\big((A_{v_1}-A_{v_2})\bar{v}_2\big), \\
  \delta \bar{\theta}|_{t=0}=0,\quad \delta \bar{v}|_{t=0}=0,
\end{cases}
\end{equation}
where $\delta \bar p :=\bar p_1 - \bar p_2$.
We rewrite this system as
\begin{equation}\label{L-BEq-del2}
\begin{cases}
  \partial_t \delta \bar{\theta} - \Delta \delta \bar{\theta} =
  \divg\left( (A_{v_1} A_{v_1}^{\textrm{T}} -  A_{v_2}A_{v_2}^{\textrm{T}})\nabla  \bar{\theta}_2 \right) - \divg\big( (\mathrm{Id}-A_{v_1} A_{v_1}^{\textrm{T}})\nabla \delta \bar{\theta}\big) , \\
  \partial_t \delta \bar{v} - \Delta \delta \bar{v}  +  \nabla \delta \overline{p} =
  (\delta \bar{\theta}) e_2 + \divg\left( (A_{v_1} A_{v_1}^{\textrm{T}} -
  A_{v_2}A_{v_2}^{\textrm{T}})\nabla  \bar{v}_2 \right)-(A_{v_1}^{\textrm{T}}-A_{v_2}^{\textrm{T}})\nabla  \overline{p}_2  \\
  \qquad \qquad \qquad \qquad\quad - \divg\big((\mathrm{Id}- A_{v_1} A_{v_1}^{\textrm{T}}) \nabla \delta \bar{v}\big) + \big(\mathrm{Id}-A_{v_1}^{\textrm{T}}\big)\nabla \delta \overline{p}, \\
  \divg \delta\bar{v} = \divg\big((A_{v_1}-A_{v_2})\bar{v}_2\big) + \divg\big((\mathrm{Id}-A_{v_1})\delta\bar{v}\big), \\
  \delta \bar{\theta}|_{t=0}=0,\quad \delta \bar{v}|_{t=0}=0.
\end{cases}
\end{equation}
Denoting by $\Lambda^{-1}:=(-\Delta)^{-\frac{1}{2}}$ and $\mathcal{R}:=\nabla \Lambda^{-1}$, we start from the first equation of system \eqref{L-BEq-del2} to get
\begin{align}\label{del-the-es1}
  & \frac{1}{2} \frac{\dd }{\dd t} \|\Lambda^{-1}\delta\bar{\theta}(t)\|_{L^2(\R^2)}^2  + \|\delta \bar{\theta}(t)\|_{L^2(\R^2)}^2 \nonumber \\
  \leq  & \left|\int_{\R^2} \Lambda^{-1}\left( (A_{v_1} A_{v_1}^{\textrm{T}} -  A_{v_2}A_{v_2}^{\textrm{T}})\nabla  \bar{\theta}_2 \right)\cdot \mathcal{R}(\delta\bar{\theta}) \dd x\right|
  + \left|\int_{\R^2} \Lambda^{-1}\big( (\mathrm{Id}-A_{v_1} A_{v_1}^{\textrm{T}})\nabla \delta \bar{\theta}\big) \cdot \mathcal{R}(\delta\bar{\theta}) \dd x \right| \nonumber \\
  \leq & \left|\int_{\R^2}\left( (A_{v_1} A_{v_1}^{\textrm{T}} -  A_{v_2}A_{v_2}^{\textrm{T}}) \,\bar{\theta}_2 \right)
  \cdot \mathcal{R}\otimes \mathcal{R}(\delta\bar{\theta}) \dd x\right|
  + \left|\int_{\R^2}\left( \nabla(A_{v_1} A_{v_1}^{\textrm{T}} -  A_{v_2}A_{v_2}^{\textrm{T}})  \;\bar{\theta}_2 \right)\cdot \mathcal{R}\Lambda^{-1}(\delta\bar{\theta}) \dd x\right| \nonumber \\
  & + \left|\int_{\R^2}\big( (\mathrm{Id}-A_{v_1} A_{v_1}^{\textrm{T}}) \,\delta \bar{\theta}\big) \cdot
  \mathcal{R}\otimes \mathcal{R}(\delta\bar{\theta}) \dd x \right|
  + \left|\int_{\R^2}\big( \nabla(\mathrm{Id}-A_{v_1} A_{v_1}^{\textrm{T}})\; \delta \bar{\theta}\big) \cdot
  \mathcal{R}\Lambda^{-1}(\delta\bar{\theta}) \dd x \right| \nonumber \\
  := &\, \mathrm{I}_1 + \mathrm{I}_2 + \mathrm{I}_3 + \mathrm{I}_4,
\end{align}
where we have suppressed the $t$-variable dependence in the formulas of $\mathrm{I}_1-\mathrm{I}_4$.
Noting that from \eqref{DxY} and \eqref{Av},
\begin{equation}\label{Av-es1}
    A_{v_1}(t,y)-A_{v_2}(t,y) =\sum_{k=1}^\infty \sum_{j=0}^{k-1} (-1)^k (C_{v_1}(t,y))^j (C_{v_2}(t,y))^{k-1-j} \int_0^t \nabla_y \delta\bar v(\tau,y)\dd y,
\end{equation}
with $C_{v_i}(t,y)= \int_0^t \nabla \bar{v}_i(\tau,y)\dd y$, $i=1,2$, and using \eqref{bar-v-Lip1}, the interpolation inequality $\|f\|_{L^{\frac{2}{1-\sigma}}(\R^2)}\leq C\|f\|_{L^2(\R^2)}^{1-\sigma} \|\nabla f\|_{L^2(\R^2)}^\sigma$, Young's inequality, we estimate the term $\mathrm{I}_1$ as
\begin{align*}
  \mathrm{I}_1 & \leq \|A_{v_1}-A_{v_2}\|_{L^{\frac{2}{1-\sigma}}} (\|A_{v_1}\|_{L^\infty} + \|A_{v_2}\|_{L^\infty}) \|\bar{\theta}_2(t)\|_{L^{\frac{2}{\sigma}}} \|\mathcal{R}^2\delta\bar{\theta}(t)\|_{L^2} \\
  & \leq C \int_0^t \|\nabla_y \delta\bar{v}(\tau,y)\|_{L^{\frac{2}{1-\sigma}}_y}\dd \tau\, \|\bar{\theta}_2(t)\|_{L^{\frac{2}{\sigma}}} \|\delta\bar{\theta}(t)\|_{L^2} \\
  & \leq \frac{1}{8} \|\delta\bar{\theta}(t)\|_{L^2}^2 + C \Big(\int_0^t \|\nabla \delta\bar{v}(\tau)\|_{L^2} \dd \tau\Big)^{2(1-\sigma)} \Big(\int_0^t \|\nabla^2 \delta\bar{v}(\tau)\|_{L^2}\dd \tau\Big)^{2\sigma}  \|\bar{\theta}_2(t)\|_{L^{\frac{2}{\sigma}}}^2 \\
  & \leq \frac{1}{8} \|\delta\bar{\theta}(t)\|_{L^2}^2 + \frac{1}{8}\Big(\int_0^t \|\nabla^2 \delta\bar{v}(\tau)\|_{L^2}^2\dd \tau\Big)
  +  C t^{\frac{1}{1-\sigma}} \Big(\int_0^t \|\nabla \delta\bar{v}(\tau)\|_{L^2}^2 \dd \tau\Big) \|\bar{\theta}_2(t)\|_{L^{\frac{2}{\sigma}}}^{\frac{2}{1-\sigma}}.
\end{align*}
For the term $\mathrm{I}_2$, observing that
\begin{equation}\label{nab-Av}
  \nabla A_{v_i}(t,y)=\sum_{k=0}^\infty (k+1) \big(C_{v_i}(t,y)\big)^k \nabla_y^2 X_{v_i}(t,y),
\end{equation}
and also
\begin{equation}\label{nab-Av2}
\begin{split}
  & \nabla A_{v_1}(t,y) - \nabla A_{v_2}(t,y) \\
  = & \sum_{k=0}^\infty (k+1) \big(C_{v_2}(t,y)\big)^k \int_0^t \nabla_y^2 \delta \bar{v}(\tau,y)\dd\tau \, + \\
  & + \sum_{k=1}^\infty \sum_{j=0}^{k-1} (k+1)\big(C_{v_1}(t,y)\big)^j (C_{v_2(t,y)})^{k-1-j} \Big(\int_0^t \nabla_y\delta\bar{v}(\tau,y)\dd\tau\Big) \nabla_y^2 X_{v_1}(t,y) ,
\end{split}
\end{equation}
and by using estimates \eqref{bar-v-Lip1}, \eqref{Avbd1}, we find that
\begin{align*}
  \mathrm{I}_2 \leq & \,C \|\Lambda^{-1}\mathcal{R}\delta\bar{\theta}(t)\|_{L^{\frac{2}{1-\sigma}}} \|\bar{\theta}_2(t)\|_{L^{\frac{2}{\sigma}}} \Big(\int_0^t \|\nabla_y^2\delta\bar{v}(\tau,y)\|_{L^2_y}\dd \tau \Big) \\
  & + C \|\Lambda^{-1}\mathcal{R}\delta\bar{\theta}(t)\|_{L^{\frac{2}{1-\sigma}}} \Big(\int_0^t \|\nabla_y \delta\bar{v}(\tau,y)\|_{L^{\frac{2}{1-\sigma}}_y}\dd\tau\Big) \|\nabla_y^2 X_{v_1}(t,y)\|_{L^{\frac{2}{\sigma}}_y} \|\bar{\theta}_2(t)\|_{L^{\frac{2}{\sigma}}} \\
  \leq &\, C t^{\frac{1}{2}} \|\Lambda^{-1}\delta\bar{\theta}(t)\|_{L^2}^{1-\sigma} \| \delta\bar{\theta}(t)\|_{L^2}^\sigma \|\bar{\theta}_2(t)\|_{L^{\frac{2}{\sigma}}} \Big(\int_0^t \|\nabla^2\delta\bar{v}(\tau)\|_{L^2}^2\dd \tau\Big)^{1/2} \\
  & + C \|\Lambda^{-1}\delta\bar{\theta}(t)\|_{L^2}^{1-\sigma} \|\delta\bar{\theta}(t)\|_{L^2}^\sigma \Big(\int_0^t \|\nabla \delta\bar{v}\|_{L^2} \dd \tau\Big)^{1-\sigma} \Big(\int_0^t \|\nabla^2 \delta\bar{v}\|_{L^2}\dd \tau\Big)^\sigma
  \|\nabla^2 X_{v_1}(t)\|_{L^{\frac{2}{\sigma}}} \|\bar{\theta}_2(t)\|_{L^{\frac{2}{\sigma}}} \\
  \leq &\, \frac{1}{8} \|\delta\bar{\theta}(t)\|_{L^2}^2 + C t^{\frac{1}{2-\sigma}}  \|\Lambda^{-1}\delta\bar{\theta}(t)\|_{L^2}^{\frac{2(1-\sigma)}{2-\sigma}} \|\bar{\theta}_2(t)\|_{L^{\frac{2}{\sigma}}}^{\frac{2}{2-\sigma}}
  \Big(\int_0^t \|\nabla^2\delta\bar{v}(\tau)\|_{L^2}^2\dd \tau\Big)^{\frac{1}{2-\sigma}} \\
  & + C t^{\frac{\sigma}{2-\sigma}}  \|\Lambda^{-1}\delta\bar{\theta}(t)\|_{L^2}^{\frac{2(1-\sigma)}{2-\sigma}} \Big(\int_0^t \|\nabla \delta\bar{v}\|_{L^2} \dd \tau\Big)^{\frac{2(1-\sigma)}{2-\sigma}} \Big(\int_0^t \|\nabla^2\delta\bar{v}\|_{L^2}^2\dd \tau\Big)^{\frac{\sigma}{2-\sigma}}
  \|\nabla^2 X_{v_1}(t)\|_{L^{\frac{2}{\sigma}}}^{\frac{2}{2-\sigma}} \|\bar{\theta}_2(t)\|_{L^{\frac{2}{\sigma}}}^{\frac{2}{2-\sigma}} \\
  \leq &\, \frac{1}{8} \|\delta\bar{\theta}(t)\|_{L^2}^2 + \frac{1}{8}\int_0^t \|\nabla^2\delta\bar{v}(\tau)\|_{L^2}^2\dd \tau +  C t^{\frac{1}{1-\sigma}} \|\Lambda^{-1}\delta\bar{\theta}(t)\|_{L^2}^2 \|\bar{\theta}_2(t)\|_{L^{\frac{2}{\sigma}}}^{\frac{2}{1-\sigma}} \\
  & + C t^{\frac{2-\sigma}{2-2\sigma}} \|\Lambda^{-1}\delta\bar{\theta}(t)\|_{L^2}\Big(\int_0^t \|\nabla \delta\bar{v}(\tau)\|_{L^2} \dd \tau\Big) \|\nabla^2 X_{v_1}(t)\|_{L^{\frac{2}{\sigma}}}^{\frac{1}{1-\sigma}} \|\bar{\theta}_2(t)\|_{L^{\frac{2}{\sigma}}}^{\frac{1}{1-\sigma}}.
\end{align*}
For $\mathrm{I}_3$, by virtue of estimates \eqref{Av}, \eqref{DxY} and \eqref{bar-v-Lip1}, we deduce
\begin{align*}
  \mathrm{I}_3 & \leq 2 \|\mathrm{Id}-A_{v_1}(t)\|_{L^\infty} \|A_{v_1}(t)\|_{L^\infty} \|\delta \bar{\theta}(t)\|_{L^2} \|\mathcal{R}^2 \delta \bar{\theta}(t)\|_{L^2} \\
  & \leq 8 \Big(\int_0^t \|\nabla \bar{v}_1(\tau)\|_{L^\infty}\dd \tau\Big) \|\delta \bar{\theta}(t)\|_{L^2}^2
  \leq \frac{1}{8} \|\delta \bar{\theta}(t)\|_{L^2}^2,
\end{align*}
where we have let the constant $c_0>0$ in \eqref{bar-v-Lip1} be such that $c_0 <\frac{1}{64}$.
For $\mathrm{I}_4$, thanks to \eqref{bar-v-Lip1}, \eqref{nab-Av}, the interpolation inequality and Young's inequality again, we infer
\begin{align*}
  \mathrm{I}_4 & \leq 2 \|\nabla A_{v_1}(t)\|_{L^{\frac{2}{\sigma}}} \|A_{v_1}(t)\|_{L^\infty} \|\delta \bar{\theta}(t)\|_{L^2} \|\mathcal{R}\Lambda^{-1}(\delta\bar{\theta})(t)\|_{L^{\frac{2}{1-\sigma}}} \\
  & \leq C  \|\delta \bar{\theta}(t)\|_{L^2}^{1+\sigma} \|\Lambda^{-1}\delta\bar{\theta}(t)\|_{L^2}^{1-\sigma} \|\nabla^2 X_{v_1}(t)\|_{L^{\frac{2}{\sigma}}} \\
  & \leq \frac{1}{8}  \|\delta \bar{\theta}(t)\|_{L^2}^2 + C \|\Lambda^{-1}\delta\bar{\theta}(t)\|_{L^2}^2\|\nabla^2 X_{v_1}(t)\|_{L^{\frac{2}{\sigma}}}^{\frac{2}{1-\sigma}} .
\end{align*}
Gathering \eqref{del-the-es1} and the above estimates on $\mathrm{I}_1-\mathrm{I}_4$, we integrate on the time interval $[0,T']$ to derive
\begin{align}
  & \|\Lambda^{-1}\delta\bar{\theta}\|_{L^\infty_{T'} (L^2)}^2 + \|\delta\bar{\theta}\|_{L^2_{T'}(L^2)}^2 \nonumber \\
  \leq & \frac{T'}{2} \|\nabla^2 \delta\bar{v}\|_{L^2_{T'}(L^2)}^2 + C\|\Lambda^{-1}\delta \bar{\theta}\|_{L^\infty_{T'}(L^2)}^2
  \big(T'^{\frac{2-\sigma}{1-\sigma}}\|\bar{\theta}_2\|_{L^\infty_{T'}(L^{\frac{2}{\sigma}})}^{\frac{2}{1-\sigma}} + T' \|\nabla^2 X_{v_1}\|_{L^\infty_{T'} (L^{\frac{2}{\sigma}})}^{\frac{2}{1-\sigma}} \big) \nonumber \\
  & + C T'^{\frac{6-5\sigma}{2-2\sigma}} \|\Lambda^{-1}\delta \bar{\theta}\|_{L^\infty_{T'}(L^2)} \|\nabla \delta\bar{v}\|_{L^\infty_{T'}(L^2)} \|\bar{\theta}_2\|_{L^\infty_{T'}(L^{\frac{2}{\sigma}})}^{\frac{1}{1-\sigma}} \|\nabla^2 X_{v_1}\|_{L^\infty_{T'} (L^{\frac{2}{\sigma}})}^{\frac{1}{1-\sigma}} \nonumber\\
  \leq & \frac{T'}{2} \|\nabla^2 \delta\bar{v}\|_{L^2_{T'}(L^2)}^2 + C\|\Lambda^{-1}\delta \bar{\theta}\|_{L^\infty_{T'}(L^2)}^2
  \big(T'^{\frac{2-\sigma}{1-\sigma}}\|\bar{\theta}_2\|_{L^\infty_{T'} (L^{\frac{2}{\sigma}})}^{\frac{2}{1-\sigma}} + T' \|\nabla^2 X_{v_1}\|_{L^\infty_{T'} (L^{\frac{2}{\sigma}})}^{\frac{2}{1-\sigma}} \big) \nonumber \\
   & + C T'^4 \|\nabla \delta\bar{v}\|_{L^\infty_{T'}(L^2)}^2   \|\nabla^2 X_{v_1}\|_{L^\infty_{T'}( L^{\frac{2}{\sigma}})}^{\frac{2}{1-\sigma}}.
\end{align}
From the continuous embedding $B^{2-\sigma}_{\frac{4}{4-\sigma},\infty}(\R^2)\hookrightarrow B^{2-\frac{3}{2}\sigma}_{\frac{4}{4-\sigma},1}(\R^2)\hookrightarrow L^{\frac{2}{\sigma}}(\R^2) $
and estimate \eqref{the-v-apEs},
we know that
\begin{equation*}
  \|\bar{\theta}_2\|_{L^\infty_T (L^{\frac{2}{\sigma}})} \leq \|\theta_2\|_{L^\infty_T (L^{\frac{2}{\sigma}})}\leq C \|\theta_2\|_{L^\infty_T (B^{2-\sigma}_{\frac{4}{4-\sigma},\infty})} <\infty,
\end{equation*}
and $\|\nabla_x^2 v_1\|_{L^\infty_T (L^{\frac{2}{\sigma}})} \leq C\|\nabla^2_x v_1\|_{L^\infty_T (B^{2-\sigma}_{\frac{4}{4-\sigma},\infty})} \leq C\|\nabla_x v_1\|_{L^\infty_T (B^{3-\sigma}_{\frac{4}{4-\sigma},\infty})}<\infty$, and
noting that
\begin{equation*}
  \nabla^2_y X_{v_1}(t,y)= \int_0^t \nabla_y^2\bar{v}_1(\tau,y)\dd \tau = \int_0^t \nabla_y X_{v_1}\cdot\nabla_x^2v_1(\tau,X_{v_1})\cdot\nabla_y X_{v_1}\dd \tau + \int_0^t \nabla_x v_1(\tau,X_{v_1})\cdot \nabla_y^2X_{v_1}(\tau,y)\dd \tau ,
\end{equation*}
we get
\begin{equation}\label{nab2-Xv}
  \|\nabla^2 X_{v_1}\|_{L^\infty_T (L^{\frac{2}{\sigma}})}\leq e^{\int_0^T \|\nabla_x v_1(\tau)\|_{L^\infty} \dd\tau} \int_0^T \|\nabla X_{v_1}(\tau)\|_{L^\infty}^2 \|\nabla_x^2 v_1(\tau)\|_{L^{\frac{2}{\sigma}}} \dd \tau <\infty,
\end{equation}
thus by letting $T'>0$ be sufficiently small so that
\begin{equation}\label{Tasum1}
  C\Big(T'^{\frac{2-\sigma}{1-\sigma}}\|\bar{\theta}_2\|_{L^\infty_T (L^{\frac{2}{\sigma}})}^{\frac{2}{1-\sigma}} + T' \|\nabla^2 X_{v_1}\|_{L^\infty_T (L^{\frac{2}{\sigma}})}^{\frac{2}{1-\sigma}} \Big) \leq \frac{1}{2},
\end{equation}
it leads to that
\begin{align}\label{del-the-es2}
  \frac{1}{2}\|\Lambda^{-1}\delta\bar{\theta}\|_{L^\infty_{T'} (L^2)}^2 + \|\delta\bar{\theta}\|_{L^2_{T'}(L^2)}^2
  \leq \frac{T'}{2} \|\nabla^2 \delta\bar{v}\|_{L^2_{T'}(L^2)}^2 + C T'^4 \|\nabla \delta\bar{v}\|_{L^\infty_{T'} (L^2)}^2   \|\nabla^2 X_{v_1}\|_{L^\infty_{T'} (L^{\frac{2}{\sigma}})}^{\frac{2}{1-\sigma}}.
\end{align}

Now we turn to the estimation of $\delta\bar{v}$. Owing to Lemma \ref{lem:Stokes}, we have
\begin{align}\label{II-decom}
  & \|\nabla \delta\bar{v}\|_{L^\infty_{T'} (L^2)} + \|(\partial_t \delta \bar{v}, \nabla^2 \delta\bar{v},\nabla\delta\bar{p})\|_{L^2_{T'} (L^2)} \nonumber \\
  \leq &\,C \|\delta\bar{\theta}\|_{L^2_{T'} (L^2)} + C \|\divg((A_{v_1} A_{v_1}^T -A_{v_2}A_{v_2}^T)\nabla \bar{v}_2)\|_{L^2_{T'} (L^2)} + C \|(A_{v_1}^T-A_{v_2}^T)\nabla \bar{p}_2\|_{L^2_{T'} (L^2)} \nonumber \\
  & + C \|\divg((\mathrm{Id}-A_{v_1}A_{v_1}^T)\nabla\delta\bar{v})\|_{L^2_{T'} (L^2)} + C\|(\mathrm{Id}-A_{v_1}^T)\nabla\delta\bar{p}\|_{L^2_{T'} (L^2)} + C\|\partial_t((A_{v_1}-A_{v_2})\bar{v}_2)\|_{L^2_{T'} (L^2)} \nonumber \\
  & + C \|\partial_t((\mathrm{Id}-A_{v_1})\delta\bar{v})\|_{L^2_{T'} (L^2)} + C \|\nabla \divg((A_{v_1}-A_{v_2})\bar{v}_2)\|_{L^2_{T'} (L^2)} + C \|\nabla\divg((\mathrm{Id}-A_{v_1})\delta\bar{v})\|_{L^2_{T'} (L^2)} \nonumber \\
  := & \,\mathrm{II}_1 + \mathrm{II}_2 + \mathrm{II}_3 + \mathrm{II}_4 + \mathrm{II}_5 + \mathrm{II}_6 + \mathrm{II}_7 + \mathrm{II}_8 + \mathrm{II}_9 .
\end{align}
For $\mathrm{II}_2$, similarly as the estimation of $\mathrm{I}_2$, from formulas \eqref{bar-v-Lip1}, \eqref{Av-es1} and \eqref{nab-Av2} we have
\begin{align*}
  \mathrm{II}_2 \leq & C \|\nabla(A_{v_1} A_{v_1}^T -A_{v_2}A_{v_2}^T)\,\cdot \nabla\bar{v}_2\|_{L^2_{T'} L^2} + C\|(A_{v_1} A_{v_1}^T -A_{v_2}A_{v_2}^T)\cdot\nabla^2 \bar{v}_2\|_{L^2_{T'} L^2} \\
  \leq & C T'^{\frac{1}{2}}\|\nabla \bar{v}_2\|_{L^\infty_{T'}(L^\infty)} \Big\|\int_0^t \nabla^2 \delta\bar{v}(\tau,y)\dd \tau\Big\|_{L^\infty_{T'}(L^2_y)} + \\
  & +  C T'^{\frac{1}{2}} \Big\|\int_0^t\nabla \delta\bar{v}(\tau,y)\dd\tau\Big\|_{L^\infty_{T'}(L^{\frac{2}{1-\sigma}}_y)}
  \Big(\|\nabla \bar{v}_2\|_{L^\infty_{T'}(L^\infty)} \|\nabla_y^2 X_{v_1}\|_{L^\infty_{T'} (L^{\frac{2}{\sigma}})} + \|\nabla^2 \bar{v}_2\|_{L^\infty_{T'} (L^{\frac{2}{\sigma}})} \Big) \\
  \leq & C T' \|\nabla \bar{v}_2\|_{L^\infty_{T'}(L^\infty)} \|\nabla^2 \delta\bar{v}\|_{L^2_{T'}(L^2)} + \\
  & + C T'^{\frac{3-\sigma}{2}} \|\nabla \delta\bar{v}\|_{L^\infty_{T'} (L^2)}^{1-\sigma} \|\nabla^2\delta\bar{v}\|_{L^2_{T'} (L^2)}^\sigma
  \Big(\|\nabla \bar{v}_2\|_{L^\infty_{T'}(L^\infty)} \|\nabla_y^2 X_{v_1}\|_{L^\infty_{T'} (L^{\frac{2}{\sigma}})} + \|\nabla^2 \bar{v}_2\|_{L^\infty_{T'} (L^{\frac{2}{\sigma}})} \Big) \\
  \leq &  \Big(\frac{1}{16} + C c_0\Big) \|\nabla^2 \delta\bar{v}\|_{L^2_{T'}(L^2)} + \\
  & + C T'^{\frac{3-\sigma}{2(1-\sigma)}} \|\nabla \delta\bar{v}\|_{L^\infty_{T'} (L^2)}
  \Big(\|\nabla \bar{v}_2\|_{L^\infty_{T'}(L^\infty)} \|\nabla^2 X_{v_1}\|_{L^\infty_{T'} (L^{\frac{2}{\sigma}})} + \|\nabla^2 \bar{v}_2\|_{L^\infty_{T'} (L^{\frac{2}{\sigma}})} \Big)^{\frac{1}{1-\sigma}},
\end{align*}
with $c_0\in ]0,\frac{1}{2}[$ the constant appearing in \eqref{bar-v-Lip1} and chosen later.
For $\mathrm{II}_3$, by using \eqref{bar-v-Lip1} and \eqref{Av-es1}, it gives
\begin{align*}
  \mathrm{II}_3 & \leq C T'^{\frac{1}{2}} \|A_{v_1}-A_{v_2}\|_{L^\infty_{T'}(L^{\frac{2}{1-\sigma}})} \|\nabla \bar{p}_2\|_{L^\infty_{T'}(L^{\frac{2}{\sigma}})} \\
  & \leq  C  T'^{\frac{1}{2}} \Big\|\int_0^t\nabla \delta\bar{v}(\tau,y)\dd\tau\Big\|_{L^\infty_{T'}(L^{\frac{2}{1-\sigma}}_y)} \|\nabla \bar{p}_2\|_{L^\infty_{T'}(L^{\frac{2}{\sigma}})} \\
  & \leq C T'^{\frac{3-\sigma}{2}} \|\nabla \delta\bar{v}\|_{L^\infty_{T'} (L^2)}^{1-\sigma} \|\nabla^2\delta\bar{v}\|_{L^2_{T'} (L^2)}^\sigma \|\nabla \bar{p}_2\|_{L^\infty_{T'}(L^{\frac{2}{\sigma}})} \\
  & \leq \frac{1}{16} \|\nabla^2 \delta\bar{v}\|_{L^2_{T'}(L^2)} + C T'^{\frac{3-\sigma}{2(1-\sigma)}} \|\nabla \delta\bar{v}\|_{L^\infty_{T'}(L^2)}
  \|\nabla \bar{p}_2\|_{L^\infty_{T'}(L^{\frac{2}{\sigma}})}^{\frac{1}{1-\sigma}}  .
\end{align*}
For $\mathrm{II}_4$ and $\mathrm{II}_5$, by virtue of formulas \eqref{Avbd1}, \eqref{bar-v-Lip1} and \eqref{nab-Av}, we deduce that
\begin{align*}
  \mathrm{II}_4 & \leq C \|\nabla(A_{v_1}A_{v_1}^T)\,\nabla\delta\bar{v}\|_{L^2_{T'}L^2} + C \|(\mathrm{Id}-A_{v_1} A_{v_1}^T)\cdot \nabla^2\delta\bar{v}\|_{L^2_{T'}(L^2)} \\
  & \leq C \|\nabla^2 X_{v_1}\|_{L^\infty_{T'} (L^{\frac{2}{\sigma}})} \|\nabla \delta\bar{v}\|_{L^2_{T'}(L^{\frac{2}{1-\sigma}})} + C \|\mathrm{Id}-A_{v_1}\|_{L^\infty_{T'} (L^\infty)} \|\nabla^2\delta\bar{v}\|_{L^2_{T'}(L^2)} \\
  & \leq C T'^{\frac{1-\sigma}{2}} \|\nabla^2 X_{v_1}\|_{L^\infty_{T'} (L^{\frac{2}{\sigma}})} \|\nabla \delta\bar{v}\|_{L^\infty_{T'} (L^2)}^{1-\sigma} \|\nabla^2 \delta\bar{v}\|_{L^2_{T'}(L^2)}^\sigma
  + C \Big(\int_0^{T'}\|\nabla \bar{v}_1(t)\|_{L^\infty}\dd t\Big) \|\nabla^2 \delta\bar{v}\|_{L^\infty_{T'} (L^2)} \\
  & \leq \Big(\frac{1}{16} + C c_0\Big)\|\nabla^2 \delta\bar{v}\|_{L^\infty_{T'} (L^2)} + C T'^{\frac{1}{2}} \|\nabla^2 X_{v_1}\|_{L^\infty_{T'} (L^{\frac{2}{\sigma}})}^{\frac{1}{1-\sigma}} \|\nabla \delta\bar{v}\|_{L^\infty_{T'} (L^2)},
\end{align*}
and
\begin{align*}
  \mathrm{II}_5 \leq 2 C \Big(\int_0^{T'}\|\nabla \bar{v}_1(t)\|_{L^\infty}\dd t\Big) \|\nabla \delta \bar{p}\|_{L^2_{T'}(L^2)} \leq 2C c_0\,  \|\nabla \delta \bar{p}\|_{L^2_{T'}(L^2)}.
\end{align*}
For $\mathrm{II}_6$, noting that
\begin{equation}\label{par-t-Av}
  \partial_t A_{v_i}(t,y) = \sum_{k=1}^\infty k (C_{v_i}(t,y))^{k-1} \,\nabla\bar{v}_i(t,y)
\end{equation}
and
\begin{equation*}
\begin{split}
  \partial_t A_{v_1}(t,y)- \partial_t A_{v_2}(t,y) %=\,& \sum_{k=1}^\infty (-1)^k k \left((C_{v_1}(t,y))^{k-1} \nabla \bar{v}_1(t,y) - (C_{v_2}(t,y))^{k-1} \nabla \bar{v}_2(t,y) \right) \\
  = \,& \nabla\delta \bar v(t,y) \sum_{k=1}^\infty (-1)^k k (C_{v_2}(t,y))^{k-1} \\
  & + \nabla \bar{v}_1(t,y)\sum_{k=2}^\infty\sum_{j=0}^{k-2} (-1)^k k (C_{v_1}(t,y))^j (C_{v_2}(t,y))^{k-2-j} \int_0^t \nabla_y \delta\bar v(\tau,y)\dd y ,
\end{split}
\end{equation*}
with $C_{v_i}(t,y)= \int_0^t \nabla \bar{v}_i(\tau,y)\dd y$, $i=1,2$, and applying estimate \eqref{bar-v-Lip1}, it follows that
\begin{align*}
  \mathrm{II}_6 & \leq C \|\partial_t(A_{v_1}-A_{2})\,\bar{v}_2\|_{L^2_{T'}(L^2)} + C \|(A_{v_1}-A_{v_2})\,\partial_t \bar{v}_2\|_{L^2_{T'}(L^2)} \\
  & \leq C T'^{\frac{1}{2}} \|\partial_t(A_{v_1}-A_{v_2})\|_{L^\infty_{T'}(L^2)} \|\bar{v}_2\|_{L^\infty_{T'}(L^\infty)}
  + C T'^{\frac{1}{2}} \|A_{v_1}-A_{v_2}\|_{L^\infty_{T'}(L^{\frac{2}{1-\sigma}})} \|\partial_t\bar{v}_2\|_{L^\infty_{T'}(L^{\frac{2}{\sigma}})} \\
  & \leq C T'^{\frac{1}{2}} \|\nabla\delta\bar{v}\|_{L^\infty_{T'} (L^2)} \|\bar{v}_2\|_{L^\infty_{T'}(L^\infty)} %( 1 +  T' \|\nabla \bar{v}_1\|_{L^\infty_{T'} L^\infty})
  + C T'^{\frac{3-\sigma}{2}} \|\nabla \delta\bar{v}\|_{L^\infty_{T'} (L^2)}^{1-\sigma} \|\nabla^2\delta\bar{v}\|_{L^2_{T'} (L^2)}^\sigma \|\partial_t\bar{v}_2\|_{L^\infty_{T'}(L^{\frac{2}{\sigma}})} \\
  & \leq \frac{1}{16}\|\nabla^2\delta\bar{v}\|_{L^2_{T'} (L^2)} + C \|\nabla \delta\bar{v}\|_{L^\infty_{T'} (L^2)} \Big(T'^{\frac{1}{2}} \|\bar{v}_2\|_{L^\infty_{T'}(L^\infty)}
  + T'^{\frac{3-\sigma}{2(1-\sigma)}} \|\partial_t\bar{v}_2\|_{L^\infty_{T'}(L^{\frac{2}{\sigma}})}^{\frac{1}{1-\sigma}}\Big).
\end{align*}
For $\mathrm{II}_7$, by virtue of \eqref{bar-v-Lip1} and \eqref{par-t-Av}, we get
\begin{align*}
  \mathrm{II}_7 & \leq C \|\partial_t A_{v_1} \,\delta\bar{v}\|_{L^2_{T'}(L^2)} + \|(\mathrm{Id}-A_{v_1})\, \partial_t\delta\bar{v}\|_{L^2_{T'}(L^2)} \\
  & \leq C T'^{\frac{1}{2}} \|\nabla\bar{v}_1\|_{L^\infty_{T'}(L^\infty)} \|\delta\bar{v}\|_{L^\infty_{T'}(L^2)} + C \|\mathrm{Id}-A_{v_1}\|_{L^\infty_{T'}(L^\infty)} \|\partial_t\delta\bar{v}\|_{L^2_{T'}(L^2)} \\
  & \leq C T' \|\nabla\bar{v}_1\|_{L^\infty_{T'}(L^\infty)} \|\partial_t\delta\bar{v}\|_{L^2_{T'}(L^2)} + C \Big(\int_0^{T'}\|\nabla_y\bar{v}_1(\tau)\|_{L^\infty}\dd \tau\Big) \|\partial_t\delta\bar{v}\|_{L^2_{T'}(L^2)} \\
  & \leq C c_0 \|\partial_t\delta\bar{v}\|_{L^2_{T'}(L^2)},
\end{align*}
where in the third line we have used the following estimate that for all $t\in [0,T']$,
\begin{equation}\label{eq:keyfact}
  \|\delta\bar{v}(t)\|_{L^2} = \|\delta\bar{v}(t)-\delta\bar{v}(0)\|_{L^2} \leq \int_0^t \|\partial_\tau \delta\bar{v}(\tau)\|_{L^2}\dd \tau \leq T'^{\frac{1}{2}}\|\partial_t \delta\bar{v}\|_{L^2_{T'}(L^2)}.
\end{equation}
Observing that $\divg((A_{v_1}-A_{v_2})\bar{v}_2) = (A_{v_1}-A_{v_2}): \nabla \bar{v}_2$ and $\divg((\mathrm{Id}-A_{v_1})\delta\bar{v})= (\mathrm{Id}-A_{v_1}):\nabla \delta\bar{v}$, and by arguing as the estimation of $\mathrm{II}_2$ and $\mathrm{II}_4$,
the terms $\mathrm{II}_8$ and $\mathrm{II}_9$ can be estimated as follows
\begin{align*}
  \mathrm{II}_8 \leq & C\|\nabla(A_{v_1}-A_{v_2}): \nabla\bar{v}_2\|_{L^2_{T'}(L^2)} + C \|(A_{v_1}-A_{v_2}):\nabla^2 \bar{v}_2\|_{L^2_{T'}(L^2)} \\
  \leq &   \Big( \frac{1}{16} +Cc_0\Big) \|\nabla^2 \delta\bar{v}\|_{L^2_{T'}(L^2)} + \\
  & + C T'^{\frac{3-\sigma}{2(1-\sigma)}} \|\nabla \delta\bar{v}\|_{L^\infty_{T'} (L^2)}
  \Big(\|\nabla \bar{v}_2\|_{L^\infty_{T'}(L^\infty)} \|\nabla^2 X_{v_1}\|_{L^\infty_{T'} (L^{\frac{2}{\sigma}})} + \|\nabla^2 \bar{v}_2\|_{L^\infty_{T'} (L^{\frac{2}{\sigma}})} \Big)^{\frac{1}{1-\sigma}},
\end{align*}
and
\begin{align*}
  \mathrm{II}_9 \leq & C \|\nabla A_{v_1} : \nabla \delta\bar{v}\|_{L^2_{T'}(L^2)} + C \|(\mathrm{Id}-A_{v_1}):\nabla^2\delta\bar{v}\|_{L^2_{T'}(L^2)} \\
  \leq & \Big(\frac{1}{16} + C c_0\Big)\|\nabla^2 \delta\bar{v}\|_{L^\infty_{T'} (L^2)}
  + C T'^{\frac{1}{2}} \|\nabla^2 X_{v_1}\|_{L^\infty_{T'} (L^{\frac{2}{\sigma}})}^{\frac{1}{1-\sigma}} \|\nabla \delta\bar{v}\|_{L^\infty_{T'} (L^2)}.
\end{align*}
Collecting formulas \eqref{II-decom}, \eqref{del-the-es2} and the above estimates on $\mathrm{II}_2-\mathrm{II}_9$, we obtain
\begin{equation}\label{eq:del-v-sum}
\begin{split}
  & \|\nabla \delta\bar{v}\|_{L^\infty_{T'} (L^2)} + \|\partial_t \delta \bar{v} \|_{L^2_{T'}(L^2)} + \frac{5}{8} \|\nabla^2 \delta\bar{v}\|_{L^2_{T'}(L^2)} + \|\nabla\delta\bar{p}\|_{L^2_{T'} (L^2)} \\
  \leq &  \big(T'^{\frac{1}{2}} + 4C c_0\big)\|\nabla^2 \delta\bar{v}\|_{L^2_{T'}(L^2)} + 2C c_0\,  \|\nabla \delta \bar{p}\|_{L^2_{T'}(L^2)}  + C c_0 \|\partial_t\delta\bar{v}\|_{L^2_{T'}(L^2)}   \\
  & + C T'^{\frac{3-\sigma}{2(1-\sigma)}}
  \Big(\|\nabla \bar{v}_2\|_{L^\infty_{T'}(L^\infty)} \|\nabla^2 X_{v_1}\|_{L^\infty_{T'} (L^{\frac{2}{\sigma}})} +\|\nabla^2 \bar{v}_2\|_{L^\infty_{T'} (L^{\frac{2}{\sigma}})} \Big)^{\frac{1}{1-\sigma}} \|\nabla \delta\bar{v}\|_{L^\infty_{T'} (L^2)}\\
  & +C T'^{\frac{3-\sigma}{2(1-\sigma)}} \Big(\|\nabla \bar{p}_2\|_{L^\infty_{T'}(L^{\frac{2}{\sigma}})}^{\frac{1}{1-\sigma}} + \|\partial_t\bar{v}_2\|_{L^\infty_{T'}(L^{\frac{2}{\sigma}})}^{\frac{1}{1-\sigma}} \Big) \|\nabla \delta\bar{v}\|_{L^\infty_{T'} (L^2)} \\
  & +  C T'^{\frac{1}{2}} \Big(\|\nabla^2 X_{v_1}\|_{L^\infty_{T'} (L^{\frac{2}{\sigma}})}^{\frac{1}{1-\sigma}} +  \|\bar{v}_2\|_{L^\infty_{T'}(L^\infty)}\Big) \|\nabla \delta\bar{v}\|_{L^\infty_{T'} (L^2)} .
\end{split}
\end{equation}
Noting that from estimates \eqref{the-v-apEs}, \eqref{DXvest}, \eqref{bar-v-Lip1}, \eqref{nab2-Xv},
\begin{align*}
  \|\nabla^2 \bar{v}_2\|_{L^\infty_T (L^{\frac{2}{\sigma}})} & \leq \|\nabla X_{v_2}\|_{L^\infty_T (L^\infty)}^2 \|\nabla^2_x v_2(t,X_{v_2})\|_{L^\infty_T (L^{\frac{2}{\sigma}})}
  + \|\nabla_x v_2(t, X_{v_2})\|_{L^\infty_T (L^\infty)} \|\nabla^2 X_{v_2}\|_{L^\infty_T (L^{\frac{2}{\sigma}})} \\
  & \leq C \|\nabla X_{v_2}\|_{L^\infty_T (L^\infty)}^2 \|\nabla v\|_{L^\infty_T(B^{3-\sigma}_{\frac{4}{4-\sigma},\infty})} + \|\nabla v_2\|_{L^\infty_T (L^\infty)}
  \|\nabla^2 X_{v_2}\|_{L^\infty_T (L^{\frac{2}{\sigma}})} <\infty,
\end{align*}
and $\|\bar{v}_2\|_{L^\infty_T (L^\infty)} \leq \|v_2\|_{L^\infty_T (L^\infty)}$, and
\begin{align*}
  & \|\nabla\bar{p}_2\|_{L^\infty_T (L^{\frac{2}{\sigma}})} + \|\partial_t\bar{v}_2\|_{L^\infty_T(L^{\frac{2}{\sigma}})} \\
  \leq &  \|\nabla_x p_2(t, X_{v_2})\|_{L^\infty_T (L^{\frac{2}{\sigma}})}
  \|\nabla X_{v_2}\|_{L^\infty_T (L^\infty)} + \|\partial_t v(t,X_{v_2})\|_{L^\infty_T (L^{\frac{2}{\sigma}})} + \|v_2(t,X_{v_2})\|_{L^\infty_T(L^{\frac{2}{\sigma}})} \|\partial_t X_{v_2}\|_{L^\infty_T (L^\infty)} \\
  \leq & C \|\nabla p_2\|_{L^\infty_T (B^{2-\sigma}_{\frac{4}{4-\sigma},\infty})}\|\nabla X_{v_2}\|_{L^\infty_T (L^\infty)} + C \|\partial_t v_2\|_{L^\infty_T(B^{2-\sigma}_{\frac{4}{4-\sigma},\infty})}
  + C \|v_2\|_{L^\infty_T (H^1)} \|v_2\|_{L^\infty_T (L^\infty)} <\infty ,
\end{align*}
we infer that by letting $T'>0$ small enough so that the positive constant $c_0$ in \eqref{bar-v-Lip1} satisfies $c_0 \leq \min \{\frac{1}{64}, \frac{1}{64 C}\}$
and also
\begin{equation}\label{Tasum2}
\begin{split}
  & T'\leq \frac{1}{256}, \quad\quad
   C T'^{\frac{3-\sigma}{2(1-\sigma)}} \Big(\|\nabla \bar{v}_2\|_{L^\infty_T (L^\infty)} \|\nabla^2 X_{v_1}\|_{L^\infty_T (L^{\frac{2}{\sigma}})} +\|\nabla^2 \bar{v}_2\|_{L^\infty_T (L^{\frac{2}{\sigma}})} \Big)^{\frac{1}{1-\sigma}} \leq \frac{1}{4}, \\
  & C T'^{\frac{3-\sigma}{2(1-\sigma)}} \Big(\|\nabla \bar{p}_2\|_{L^\infty_T (L^{\frac{2}{\sigma}})}^{\frac{1}{1-\sigma}} + \|\partial_t\bar{v}_2\|_{L^\infty_T(L^{\frac{2}{\sigma}})}^{\frac{1}{1-\sigma}} \Big)
  + C T'^{\frac{1}{2}} \Big(\|\nabla^2 X_{v_1}\|_{L^\infty_T (L^{\frac{2}{\sigma}})}^{\frac{1}{1-\sigma}} +  \|\bar{v}_2\|_{L^\infty_T(L^\infty)}\Big) \leq \frac{1}{4},
\end{split}
\end{equation}
the whole right-hand side of inequality \eqref{eq:del-v-sum} can be absorbed by the left-hand side, which leads to $\|\nabla \delta\bar{v}\|_{L^\infty_{T'} (L^2)}=\|\nabla^2\delta\bar{v}\|_{L^2_{T'} (L^2)}\equiv 0$,
and also by estimate \eqref{del-the-es2}, $\|\Lambda^{-1}\delta\bar{\theta}\|_{L^\infty_{T'}(L^2)} = \|\delta\bar{\theta}\|_{L^2_{T'}(L^2)} \equiv 0$. Since $\delta\bar{v}|_{t=0}\equiv 0$,
we conclude that $\delta\bar{v}=0$ and $\delta\bar\theta=0$ a.e. on $\R^2\times [0,T']$. Repeating the above procedure, we can further prove $\delta\bar{v}=\delta\bar\theta=0$ a.e. on $[T',2T']$, $[2T',3T']$, $\cdots$,
where $T'>0$ is a small constant depending only on the time $T$ and the norms of $(\theta_i,v_i,p_i)$ in estimate \eqref{the-v-apEs} (similar to conditions \eqref{Tasum1} and \eqref{Tasum2}),
hence we finally get $\delta\bar{v}=\delta\bar\theta=0$ and also $X_{v_1}=X_{v_2}$ a.e. on $\R^2\times[0,T]$.
Going back to the Eulerian coordinates implies that $(\mu_1,\theta_1,v_1)=(\mu_2,\theta_2,v_2)$ a.e. on $\R^2\times[0,T]$ and the uniqueness of system \eqref{BoussEq} is followed.

\section{Appendix}\label{sec:append}

\begin{proof}[Proof of Lemma \ref{lem:prodEs}]
(1) For every $q\geq -1$, Bony's decomposition yields that
\begin{align}\label{eq:decom}
  \Delta_q (v\cdot\nabla \theta) = & \sum_{k\geq-1,|k-q|\leq 4} \Delta_q(S_{k-1}v\cdot\nabla \Delta_k \theta)
  + \sum_{k\geq -1,|k-q|\leq 4} \Delta_q (\Delta_k v\cdot\nabla S_{k-1} \theta) \nonumber \\
  & + \sum_{k\geq q-3} \Delta_q (\Delta_k v\cdot\nabla \widetilde{\Delta}_k\theta)
  := \,I_q + II_q + III_q,
\end{align}
where $\Delta_k$ and $S_{k-1}$ are Littlewood-Paley operators introduced in \eqref{LPop} and $\widetilde{\Delta}_k:= \Delta_{k-1}+\Delta_k +\Delta_{k+1}$.
Taking advantage of H\"older's inequality, Bernstein's inequality and the fact that $\|v\|_{L^2}^2 = \sum_{k\geq -1}\|\Delta_k v\|_{L^2}^2$, we get
\begin{align*}
  2^{-qs}\|I_q\|_{L^p(\R^2)} & \leq  C_0  \sum_{k\geq-1,|k-q|\leq 4} 2^{-q s} \|S_{k-1}v\|_{L^\infty(\R^2)} \|\nabla \Delta_k \theta\|_{L^p(\R^2)} \\
  & \leq C_0  \sum_{k\geq-1,|k-q|\leq 4} \Big(\sum_{-1\leq l\leq k-2} \|\Delta_l v\|_{L^\infty(\R^2)}\Big) 2^{k(1-s)}\| \Delta_k \theta\|_{L^p(\R^2)} \\
  & \leq C_0  \sum_{k\geq-1,|k-q|\leq 4} \Big(\|\Delta_{-1}v\|_{L^2(\R^2)} + \sum_{0\leq l\leq k+2} 2^l\|\Delta_l v\|_{L^2(\R^2)}\Big)2^{k(1-s)}\| \Delta_k \theta\|_{L^p(\R^2)} \\
  & \leq C_0  \sum_{k\geq-1,|k-q|\leq 4} \bigg(\|v\|_{L^2(\R^2)} + \Big(\sum_{0\leq l\leq k+2} \|\Delta_l \nabla v\|_{L^2(\R^2)}^2\Big)^{1/2} \bigg) \sqrt{k+2}\, 2^{k(1-s)}\| \Delta_k \theta\|_{L^p(\R^2)}\\
  & \leq C_0 \left(\|v\|_{L^2(\R^2)} + \|\nabla v\|_{L^2(\R^2)}\right) \Big(\sup_{k\geq -1} 2^{k(1-s)} \sqrt{k+2}\|\Delta_k\theta\|_{L^p(\R^2)}\Big),
\end{align*}
and
\begin{align*}
  & 2^{-qs}\|II_q\|_{L^p(\R^2)} \leq  C_0  \sum_{k\geq-1,|k-q|\leq 4} 2^{-q s} \|\Delta_k v\|_{L^\infty(\R^2)} \|\nabla S_{k-1}\theta\|_{L^p(\R^2)} \\
  & \leq  C_0  \sum_{k\geq-1,|k-q|\leq 4} \big(\|\Delta_{-1} v\|_{L^\infty(\R^2)} +\|(\mathrm{Id}-\Delta_{-1})\Delta_k v\|_{L^\infty(\R^2)}  \big)
  \bigg(2^{-k s} \sum_{-1\leq l\leq k-2} 2^l \| \Delta_l\theta\|_{L^p(\R^2)}\bigg) \\
  & \leq  C_0  \sum_{k\geq-1,|k-q|\leq 4} \big(\|v\|_{L^2(\R^2)} +\|(\mathrm{Id}-\Delta_{-1})\Delta_k \nabla v\|_{L^2(\R^2)}  \big)
  \bigg(2^{-k s} \sum_{-1\leq l\leq k-2} 2^{ls} \| \theta\|_{B^{1-s}_{p,\infty}(\R^2)}\bigg) \\
  & \leq C \left(\|v\|_{L^2(\R^2)} + \|\nabla v\|_{L^2(\R^2)}\right) \|\theta\|_{B^{1-s}_{p,\infty}(\R^2)}.
\end{align*}
By using the divergence-free property of $v$, we similarly obtain
\begin{align*}
  2^{-qs}\|III_q\|_{L^p(\R^2)} & \leq  C_0  \sum_{k\geq-1,k\geq q-3} 2^{q(1-s)} \|\Delta_k v\|_{L^\infty(\R^2)} \|\widetilde{\Delta}_k\theta\|_{L^p(\R^2)} \\
  & \leq C_0 \big(\|v\|_{L^2(\R^2)} + \|\nabla v\|_{L^2(\R^2)} \big)
  \sum_{k\geq-1,k\geq q-3} 2^{(q-k)(1-s)} 2^{k(1-s)}\|\widetilde{\Delta}_k\theta\|_{L^p(\R^2)} \\
  & \leq C \left(\|v\|_{L^2(\R^2)} + \|\nabla v\|_{L^2(\R^2)}\right) \|\theta\|_{B^{1-s}_{p,\infty}(\R^2)}.
\end{align*}
Gathering the above estimates leads to estimate \eqref{eq:prodEs}, as desired.
\vskip0.1cm

(2) In order to prove the first inequality of \eqref{eq:prodEs2}, we also have the splitting \eqref{eq:decom}, and by arguing as above, we deduce that for every $q\geq -1$,
\begin{align*}
  2^{qs} \|I_q\|_{L^p(\R^2)} \leq  C_0  \sum_{k\geq-1,|k-q|\leq 4} 2^{q s} \|S_{k-1}v\|_{L^{2p}(\R^2)} \|\nabla \Delta_k \theta\|_{L^{2p}(\R^2)} \leq C_0 \|v\|_{L^{2p}(\R^2)} \|\theta\|_{B^{1+s}_{2p,\infty}(\R^2)} ,
\end{align*}
and
\begin{align*}
  2^{qs}\|II_q\|_{L^p(\R^2)} \leq C_0 \sum_{k\geq -1,|k-q|\leq 4} 2^{qs}\|\Delta_k v\|_{L^{2p}(\R^2)} \|S_{k-1}\nabla \theta\|_{L^{2p}(\R^2)}
  \leq C  \|v\|_{B^s_{2p,\infty}(\R^2)} \|\nabla\theta\|_{L^{2p}(\R^2)},
\end{align*}
and
\begin{align*}
  2^{qs} \|III_q\|_{L^p(\R^2)} & \leq C_0  \sum_{k\geq-1,k\geq q-3} 2^{q(1+s)} \|\Delta_k v\|_{L^{2p}(\R^2)} \|\widetilde{\Delta}_k\theta\|_{L^{2p}(\R^2)} \\
  & \leq C_0 \Big(\sum_{k\geq-1,k\geq q-3} 2^{(q-k)(1+s)}\Big) \|v\|_{L^{2p}(\R^2)} \|\theta\|_{B^{1+s}_{2p,\infty}(\R^2)} \leq C_0  \|v\|_{L^{2p}(\R^2)} \|\theta\|_{B^{1+s}_{2p,\infty}(\R^2)}  .
\end{align*}
Hence, by collecting the above estimates we conclude the first inequality of \eqref{eq:prodEs2}, and then from the continuous embedding $B^s_{2p,\infty}\hookrightarrow L^{2p}$ and $B^{1+s}_{2p,\infty}\hookrightarrow W^{1,2p}$,
the second inequality of \eqref{eq:prodEs2} is directly followed.

(3) The proof of inequality \eqref{eq:prodEs3} can be directly deduced from the Bony's paraproduct estimates as above, and we here omit the details.
\end{proof}
\vskip0.2cm

\textbf{Acknowledgements.}
P. B. Mucha was partially supported by National Science Centre grant  No2018/29/B/ST1/00339 (Opus).
L. Xue was partially supported by National Natural Science Foundation of China (Grants Nos. 11671039 and 11771043).

\end{document}